\documentclass[a4paper,11pt,reqno,noindent]{amsart}
\usepackage[dvipsnames]{xcolor}
\usepackage{amsmath,amsfonts,amssymb,amsthm,graphics,enumerate,mathtools}
\usepackage{hyperref}
\hypersetup{
    colorlinks=true,
    linkcolor=blue,
    citecolor=red,
    urlcolor=blue,
    pdfborder={0 0 0}
}
\usepackage{graphicx,bbm,fullpage}
\usepackage{dsfont}
\usepackage{tikz}
\usetikzlibrary{shapes.geometric, arrows}

\tikzstyle{process} = [rectangle, minimum width=3cm, minimum height=1cm, text centered, draw=black, text width=3cm]
\tikzstyle{arrow} = [thick,->,>=stealth]

\usepackage{cleveref}
  \crefname{theorem}{Theorem}{Theorems}
  \crefname{thm}{Theorem}{Theorems}
  \crefname{lemma}{Lemma}{Lemmas}
  \crefname{lem}{Lemma}{Lemmas}
  \crefname{remark}{Remark}{Remarks}
  \crefname{prop}{Proposition}{Propositions}
  \crefname{proposition}{Proposition}{Propositions}
  \crefname{notation}{Notation}{Notations}
  \crefname{claim}{Claim}{Claims}
  \crefname{observation}{Observation}{Observations}
  \crefname{defn}{Definition}{Definitions}
  \crefname{corollary}{Corollary}{Corollaries}
  \crefname{section}{Section}{Sections}
  \crefname{figure}{Figure}{Figures}
  \crefname{exercise}{Exercise}{Exercises}
    \crefname{assumption}{Assumption}{Assumptions}

\newtheorem{thm}{Theorem}[section]

\newtheorem{claim}[thm]{Claim}
\newtheorem{lemma}[thm]{Lemma}
\newtheorem{corollary}[thm]{Corollary}

\newtheorem{proposition}[thm]{Proposition}
\newtheorem{defn}[thm]{Definition}

\numberwithin{equation}{section}
\newtheorem*{h1}{Assumption (H1): Transition probabilities}
\newtheorem*{h2}{Assumption (H2): Traveling far}

\theoremstyle{definition}
\newtheorem{remark}[thm]{Remark}

\newcommand{\id}{\operatorname*{id}}

\def\P{\mathbb{P}}
\def\Prob{\P}
\def\E{\mathbb{E}}

\def\Z{\mathbb{Z}}
\def\N{\mathbb{N}}

\newcommand{\bbZ}{\mathbb{Z}}

\def\eps{\varepsilon}

\DeclareMathOperator{\Id}{Id}
\DeclareMathOperator{\var}{Var}

\DeclareMathOperator{\cov}{Cov}

\newcommand{\tld}[1]{\tilde{#1}}

\def\Real{{\mathbb{R}}}
\def\Expec{{\mathbb{E}}}
\newcommand{\diff}{\mathop{}\!\mathrm{d}}
\newcommand{\Ft}[1]{{\widehat{#1}}}

\def\tauclo{\tau_{\rm clo}}
\def\taureg{\tau_{\rm reg}}

\newcommand{\Hyp}{\mcal{H}}
\newcommand{\Hypone}[1]{\hyperref[h1]{\mbf{H}_1(#1)}}
\newcommand{\Hyptwo}[1]{\hyperref[h2]{\mbf{H}_2(#1)}}

\newcommand{\mcal}[1]{{\mathcal{#1}}}
\newcommand{\mbf}[1]{{\mathbf{#1}}}

\newcommand{\Ss}{\mcal{S}}

\author[D. Elboim]{Dor Elboim}
\address[Dor Elboim]{Stanford University, Department of Mathematics, Stanford CA 94305}
\email{dorelboim@gmail.com}
\author[A. Gloria]{Antoine Gloria}
\address[Antoine Gloria]{Sorbonne Universit\'e, Universit\'e Paris Cit\'e, CNRS, Laboratoire Jacques-Louis Lions, LJLL, F-75005 Paris, France  \& Universit\'e Libre de Bruxelles, D\'epartement de Math\'ematique, 1050~Brussels, Belgium}
\email{antoine.gloria@sorbonne-universite.fr}
\author[F. Hern\'andez]{Felipe Hern\'andez}
\address[Felipe Hern\'andez]{Penn State University, Department of Mathematics, State College PA 16803}
\email{felipeh@psu.edu}

\title{Diffusivity of the Lorentz mirror walk in high dimensions}
\date{\today}

\begin{document}

\begin{abstract}
Place random mirrors at vertices of $\Z^d$ (for $d\ge 2$) independently at density $p\in [0,1]$, and shoot a ray of light in that environment. The ray changes direction when it hits a mirror, and is therefore slowed down by bouncing around -- this is the Lorentz mirror walk. The main question of interest  is whether  most trajectories are localized (finite) or delocalized (infinite), and, in the latter case, at what speed light goes to infinity. The consensus is that all trajectories should be finite when $d=2$ for any density $p>0$,  whereas in dimensions $d\ge 3$ most trajectories should be infinite (at least for small $p$), and the speed be diffusive. In this article we establish that for all $d\ge 4$, most trajectories do not localize before super-polynomial time in units of $p^{-1}$ -- more precisely, trajectories behave diffusively at times $t \simeq p^{-M}$ for all $M>1$. Combined with the strategy developed by the first author and Sly for the interchange model, this should imply diffusion in infinite time (for $p>0$ small enough) in dimensions $d\ge 5$.

\end{abstract}

\begin{abstract}
In the Lorentz mirror walk in dimension $d\ge 2$, mirrors are randomly placed on the vertices of $\mathbb Z ^d$ at density $p\in [0,1]$. A light ray is then shot from the origin and deflected through the various mirrors in space. The object of study is the random trajectory obtained in this way, and it is of upmost interest to determine whether these trajectories are localized (finite) or delocalized (infinite).
A folklore conjecture states that for $d=2$ these trajectories are finite for any density $p>0$, while in dimensions $d\ge 3$ and for $p>0$ small enough some trajectories are infinite.
In this paper we prove that for all dimensions $d\ge 4$ and any small density $p$, the trajectories behave diffusively at all polynomial time scales $t\approx p^{-M}$ with $M>1$, and in particular, they do not close by this time.
\end{abstract}

\maketitle
\setcounter{tocdepth}{1}

\tableofcontents

\section{Introduction and main result}

\subsection{Transport in heterogeneous media}

 In this paper we study the long-time behavior of the Lorentz mirror walk, which is the trajectory traced by a ray of light traveling along edges in $\bbZ^d$ and deflected by randomly placed mirrors.  This is a simple example of transport in a heterogeneous medium.
For classical and quantum dynamical models in heterogeneous environments, one expects three universal behaviors \cite{PhysRev.109.1492,Spencer-talk}: \emph{kinetic transport} (at time $t$, the particle is at distance $\simeq t$), \emph{localization} (the particle gets trapped in a bounded region), and \emph{diffusive transport} (at time $t$, the particle is at distance $\simeq \sqrt t$).  These models include the Lorentz mirror walk, the Lorentz gas, stochastic acceleration, and classical and quantum waves in periodic or random media.
Whereas localization and ballistic transport have been proved in some regimes \cite{MR4255829,MR4328425,MR2024819,MR3430256,MR696803,MR3501794,MR1405959,MR3982871}, diffusive transport has remained much more elusive\footnote{In models with an initial kinetic transport regime.}, and only proved for short times \cite{MR496299,MR471824,MR597029,MR833271,MR2207646,MR4156218,MR4314129,MR2413135,MR4838987}.

\medskip

The motion of a particle in a \emph{random} medium usually starts with a kinetic part, and then either gets trapped or displays diffusive transport at larger times. This transition from kinetic to diffusive transport is a universal phenomenon that is however still poorly understood from the mathematical perspective.
In all the models for which it has been established, diffusive transport is proved for relatively short times and by comparison to models without memory, for which the disorder is new at every time. The difficulty is to prove the closeness of the two models for times that are larger than the kinetic time. This entails what we could call the \emph{perturbative diffusive regime}. The more complex the model, the harder the proof. For the Lorentz mirror walk, the proof is two pages long and relies on the Kesten-Papanicolaou argument (see for instance \cite{Piernot} or Section~\ref{sec:KP}), whereas for quantum diffusion this is already hard analysis \cite{MR2413135,MR4838987} because the Schr\"odinger flow in a random potential simply never coincides with a model without memory.
The long-standing bottleneck is the following: whereas diffusive transport is mostly established using independence, models of interest are never independent at large times, and there is no hope to get \emph{non-perturbative diffusion} without addressing memory effects.

\medskip

In this contribution we establish \emph{non-perturbative diffusive transport} for the Lorentz mirror walk up to super-algebraic times measured in units of $p^{-1}$ in dimensions $d\ge 4$, the first such result for a model with a kinetic time scale.

\medskip

Our approach takes inspiration from the work by the first author and Sly \cite{elboim2022infinite} on the so-called interchange model  and on the coupling introduced by Lutsko and T\'oth in \cite{MR4156218}. It amounts to showing that memory effects do not destroy  long-time diffusive transport. Combined with the strategy of \cite{elboim2022infinite}, the present result is expected to imply (annealed) diffusive transport for all times in dimensions $d\ge 5$, provided $p>0$ is small enough.
We believe the methods of this paper could be used to study more complex models, such as the Lorentz gas.

\subsection{Diffusivity of the Lorentz mirror walk}

Let us introduce precisely the Lorentz mirror walk on $\Z^d$. The random  environment is a field of matchings $\mathbb M:\bbZ^d\to \mbf{M}_d$ on the set of directions $\{\pm e_j\}_{j=1}^d$ which are axis-aligned unit vectors, where
\begin{equation*}
\mbf{M}_d = \{m: \{\pm e_j\} \to \{\pm e_j\} |
m(-m(v)) = -v, m(e_j) \neq -e_j\}.
\end{equation*}
In particular we require $m(-m(v)) = -v$ (which makes the dynamics time-reversible), and avoid head-on collisions (allowing such collisions would readily imply that all walks are localized) so we enforce $m(v)\not=-v$ (which makes the model irrelevant in dimension $d=1$).
We assume that $\mathbb M$ has independent entries $\mathbb M_x$
chosen so that with probability $1-p$, the matching is $\mathbb M_x=\id$, and with probability $p$, the matching is called a \emph{mirror}, a matching uniformly chosen\footnote{Note that $\id$ is an admissible mirror too, so that, all in all, the probability that $m_x=\id$ is indeed $1-p+\frac p{|\mbf{M}_d|}$, $|\mbf{M}_d|=\prod_{j=1}^{d}(2(d-j)+1)$. } in $\mbf{M}_d$.

The Lorentz walk $W_L(t)=(X_L(t),V_L(t))$, $t\in \N$, starting at position $X_L(0)=0$ with velocity $V_L(0)=e_1$ is defined by the update rule
\begin{equation}
\label{eq:Lorentz mirror-evol}
\begin{split}
X_L(t+1) &= X_L(t) + V_L(t), \\
V_L(t+1) &= \mathbb M_{X(t+1)}(V_L(t)).
\end{split}
\end{equation}
Our main result establishes annealed diffusion up to super-polynomial time measured in units of $p^{-1}$.
\begin{thm}\label{th:diffusive}
Let $d\ge 4$. For all $p$ sufficiently small, we have for all times $p^{-5/4}\le T\le e^{\log ^2(1/p)}$
\[
\Big|\frac p T\,\E\Big[\|X_L(T)\|_2^2\Big]-\frac{2d-1}{d-1}
    \Big| \le p^{1/9}.
\]
\end{thm}
The timescale $e^{\log^2(1/p)}$ and the exponent $p^{1/9}$ are not sharp.  The same argument we present could be pushed to reach a timescale on the order $e^{p^{-\eps}}$ for some $\eps>0$.  Reaching infinite times in dimension $d\geq 5$ would require one to combine the present argument with an adaptation of the ideas of~\cite{elboim2022infinite}.

In terms of elements of proof, the closest related works are those by Sly and the first author \cite{elboim2022infinite} on the interchange model and that by Lutsko and T\'oth \cite{MR4156218}  on the Lorentz gas.

The work \cite{MR4156218} by Lutsko and T\'oth  establishes diffusive behavior of the Lorentz gas beyond the Boltzmann-Grad limit and the Kesten-Papanicolaou time-scale \cite{MR597029}. In terms of the scaling of the Lorentz walk, this would correspond to diffusion up to times much smaller than $p^{-d}$ in dimension $d\ge 3$.  The proof introduces a subtle coupling of the Lorentz gas with the Markov flight that maximizes the set of times at which their velocities coincide. This very coupling will be crucially used in the present contribution in a form adapted to the Lorentz walk. 

The work by Sly and the first author \cite{elboim2022infinite} considers the cyclic walk in the interchange model, another self interacting walk on $\bbZ^d$ which evolves deterministically given a random environment and which exhibits diffusive behavior at long times.
One difficulty in the mirror model that is not present in the interchange model is the fact that there is a kinetic timescale at which the walk is not diffusive but ballistic.  This means one must keep track of some  information at the kinetic scale $p^{-1}$ (such as its velocity and the precise location of the mirrors in its kinetic neighborhood) in addition to large-scale diffusive behavior of the walk.  In Section~\ref{sec:proof-ideas} below we explain further the additional difficulties that appear in the analysis of the mirror walk.
Another difference with \cite{elboim2022infinite} is that in the current paper we only consider polynomial time scales (or slightly more) which makes the proof shorter.  The fact that we do not obtain a result for infinite times also allows us to obtain a result in the critical\footnote{That is, critical with respect to the strategy of proof -- $d=2$ is the physical critical dimension.} dimension $d=4$.

One may also view the interchange model and the mirror model as classical analogues of quantum processes.  In particular, the interchange model can be seen as a classical analogue of random band matrix models, where there has been a flurry of activity and new results.  Superpolynomial analogues of the infinite time result in~\cite{elboim2022infinite} have recently been established for these random band matrix models (see~\cite{yang2021delocalization,yang2022delocalization,yang2025delocalization} for a sample of a few such results).  Surprisingly, a superpolynomial result was even obtained in~\cite{dubova2025delocalization} for dimension $d=2$.
A quantum analogue of the mirror model is the Anderson tight-binding model in which the dynamics has a kinetic length scale analogous to $p^{-1}$.  The presence of the kinetic length scale seems to be a serious obstacle to obtaining strong diffusive results for the Anderson model.   In the case of the Anderson model, relatively weak diffusive results have thus been obtained~\cite{erdHos2007quantum, MR2413135, MR4838987} so far, and a quantum analogue of Theorem~\ref{th:diffusive} seems out of reach of current methods.

\subsection{Ideas in the proof}
\label{sec:proof-ideas}

At the core of our argument is a multi-scale argument that applies generally to random walks in high dimensions with local self interactions.  The fundamental idea is that, so long as  walk is \textit{diffusive}, it should not interact with itself too often for the same reason that two Brownian motions in high dimensions do not intersect.
More precisely, at any given time $t$ there should be a positive probability that the walk never interacts with its history before time $t$.
This implies a ``concatenation'' property that essentially ensures that the law of the endpoint $X(2t)$ is similar to the sum of two
independent copies $X(t)+X'(t)$, and such a concatenation property allows us to propagate the diffusive estimates.  In the case of this paper
we are able to reduce the diffusive hypotheses to two simple properties: (1) an anticoncentration inequality $\hyperref[h1]{\bf{(H1)}}$ for the endpoint distribution of the walk on the scale $p^{-1}$, and (2) a tail bound $\hyperref[h2]{\bf{(H2)}}$ on the deviation $|X(s)-X(t)|$.
This multiscale argument is a version of renormalization and uses the stability of diffusive behavior to absorb the error arising from the failure of the exact Markov property.  This is
reminiscent of the renormalization approach to the central limit theorem which relies on the stability of the Gaussian fixed point under the map $X\mapsto \frac{1}{\sqrt{2}} (X+X')$ (for a nice discussion, see~\cite{ott2023note}).  A version of the multiscale argument described above was implemented in~\cite{elboim2022infinite}, though we have streamlined the presentation and the inductive hypotheses (focussing on super-polynomial times).

The challenge one faces in applying such a multi-scale argument to the mirror walk is that the mirror walk only has diffusive behavior on scales much larger than $p^{-1}$, but the events that break the Markov property of the walk (recollisions with mirrors) are determined at the much finer lattice scale.  This results in the complication that the inductive argument must weave a ``local'' lattice-scale analysis of the walk with a ``global'' analysis that keeps track of the large-scale diffusive features.  Much of the local analysis is highly technical and based on intricate casework, but we can single out two ideas to explain here at a high level.

The first idea is a careful definition of the coupling between the (Markovian) ``driving walk''  and the mirror walk, which ensures that as often as possible the driving walk and the driven walk share the same velocity.  In particular, this coupling allows the driven walk to correct
any deviation in the velocity from the driving walk as soon as a ``fresh'' mirror is discovered (leading to~\eqref{eq:prob of coupling} below). The coupling we consider is the adaptation to the Lorentz walk of the coupling introduced in \cite{MR4156218} by Lutsko and T\'oth for the Lorentz gas. This reduces our task to ensuring that the walk has plenty of opportunities to discover fresh mirrors and that self-interactions are rare enough that the total accumulated displacement
between the walks is sub-diffusive.   The key to controlling the self-interactions is to understand the following local obstruction: if the walk at time $t$ has velocity $V(t)$ and $X(t)+sV(t)$ is the location of a previously discovered mirror for some $s$ on the kinetic scale, $s\lesssim p^{-1}$, then there will likely be a self-interaction soon.  If this is the case we say that the walk fails to be \textit{locally relaxed}.

The second idea is a way to show that the set of locally relaxed times is dense.  More precisely, we introduce a set of times $\mcal{S}$ (see~\eqref{eq:S} below) which is the set of times at which the walk could possibly discover a mirror that would change the direction of the walk to head towards a previously discovered mirror, thus causing the walk to fail to be locally relaxed.  We are able to show that the set of times $\mcal{S}$ is \textit{sparse} so long as the set of discovered mirrors is never too dense in a cube of side length $p^{-1}$ (see Lemma~\ref{lem:sparsity}).  This sparsity estimate is what allows us to bridge scales -- it allows us to apply diffusive estimates on the walk, which are only valid at scales larger than $p^{-1}$, to control lattice scale events.

\subsection{Organization of the Paper}
In Section~\ref{sec:bigdriving} we set up a coupling between the mirror walk and a driving Markovian walk.  In Section~\ref{sec:hypotheses} we state the inductive hypotheses we shall work with, and describe the precise structure of the induction argument.
The rest of the paper is then dedicated to the implementation of this strategy.

\medskip

Throughout the paper, $c>0$ denotes a universal constant that may  vary from line to line, but remains independent of $p>0$ (provided $p>0$ is chosen small enough).  When not otherwise specified, we use
the $\ell^\infty$ norm on $\bbZ^d$.

\section{Driving process and regenerated walk}
\label{sec:bigdriving}
In this section we introduce another way to define the Lorentz mirror walk, where mirrors are discovered on the fly along the trajectory, rather than being set at the initial time.

\subsection{The driving process}\label{sec:driving}

The ``driving process''  is the sequence $\{(\tilde{X}(t), \tilde{V}(t),\mathds1 \{ t\in \mcal{T}\},\tilde{m}(t))\}_{t\ge 0}$ of positions $\tld{X}(t)$, velocities $\tld{V}(t)\in\{\pm e_j\}_{j=1}^d$, Bernoulli variables $\mathds 1 \{t\in \mathcal T\}$, and matchings $\tld{m}(t)\in \mbf{M}_d$ denoted by $\{\tilde{W}(t)\}_{t\ge 0}$ and defined as follows.

We fix the initial data $\tld{X}(0)\in \mathbb Z ^d$ and $\tld{V}(0)\in\{\pm e_j\}_{j=1}^d$.

Sample a random subset $\mcal{T}\subset \mathbb{N}=\{0,1,\dots \}$, where each $t\in \mathbb N$ is in the set $\mathcal T$ independently with probability $p$. These are the times in which we ``discover a new mirror". For $t\notin \mcal{T}$ we let $\tld{m}(t) = {\rm id}$ be the identity matching, and for $t\in\mcal{T}$ we let $\tld{m}(t)$ be a uniformly chosen matching in $\mbf{M}_d$. These random uniform matchings are taken independent of each other and of the random set $\mathcal T$.

We define the velocities and the position using the inductive rule $\tilde{V}(t) = \tilde{m}(t)[\tld{V}(t-1)]$ and $\tilde{X}(t) =\tilde{X}(t-1) + \tilde{V}(t-1)$.  Note that $\{\tilde{X}(t)\}_{t\ge 0}$ is simply a non-backtracking random walk that proceeds straight with probability $1-p$ and proceeds in each of the  $2d-1$ directions, not including where it came from, with probability $p/(2d-1)$.

\medskip

\begin{figure}
\begin{tabular}{lcr}
(a)
\begin{tikzpicture}
\draw [dashed] (-1,0)--(1,0);
\draw [dashed] (0,-1)--(0,1);
\draw [thick, ->, gray] (0,1) -- (0,0) -- (1,0);
\draw [thick, ->] (0,-1) -- (0,0);
\end{tikzpicture}
&\quad \quad\quad\quad\quad\quad
&
(b)
\begin{tikzpicture}
\draw [dashed] (-1,0)--(1,0);
\draw [dashed] (0,-1)--(0,1);
\draw [thick, ->, gray] (-1,0) -- (0,0) -- (1,0);
\draw [thick, ->] (0,-1) -- (0,0);
\end{tikzpicture}
\end{tabular}
\caption{The two ways a walk can interact with its past (shown in gray).  In case (a), the walk visits a location where a mirror was previously discovered.  In case (b), the walk attempts to place a mirror at a location where the walk previously ruled out the presence of a mirror.}
\label{fig:twocases}
\end{figure}
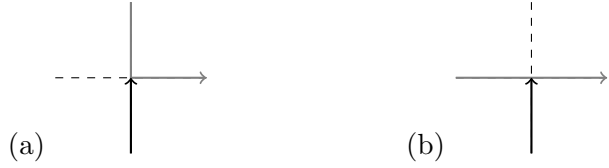

Next, we define a walk $\Phi(\tld{W})=(X,V,m)$ that is driven by $\tld{W}$. We fix the initial conditions $X(0)\in \mathbb Z ^d$ and $V(0)\in\{\pm e_j\}_{j=1}^d$, and consider the following inductive rules. We let $m(0)=\tilde{m}(0)$ and for $t\ge 1$ we set $X(t) = X(t-1)+V(t-1)$ and $V(t)=m(t)[V(t-1)]$ where $m(t)$ is defined according to the following rules:
\begin{enumerate}
\item If $X(t) = X(s)$ for some $s\in [0,t)$, then we set $m(t) = m(s)$.
\item If $X(t)\not= X(s)$ for all $s\in [0,t)$ and $V(t-1)=\tld{V}(t-1)$, then we set
$m(t)=\tld{m}(t)$.
\item If $X(t)\not= X(s)$ for all $s\in [0,t)$ and $V(t-1)\not= \tld{V}(t-1)$, and $\tilde{m}(t)[\tilde{V}(t-1)]=-V(t-1)$, then we also set $m(t)=\tld{m}(t)$.
\item Otherwise, if $X(t)\not= X(s)$ for all $s\in [0,t)$ and $V(t-1)\not= \tld{V}(t-1)$ and $\tilde{m}(t)[\tilde{V}(t-1)]\neq -V(t-1)$, then we define $m(t)$ as follows. We set $m(t)[V(t-1)]=\tilde{m}(t)[\tilde{V}(t-1)]$, $m(t)[\tilde{V}(t-1)]=\tilde{m}(t)[V(t-1)]$ and $m(t)[v]=\tilde{m}(t)[v]$ for any $v\notin \big\{ V(t-1),\tilde{V}(t-1),\tilde{m}(t)[V(t-1)],\tilde{m}(t)[\tilde{V}(t-1)]\big\}$.  A cartoon of this rule is illustrated in Figure~\ref{fig:mirror-coupling}.
\end{enumerate}

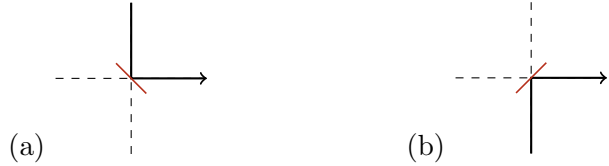
\begin{figure}
\centering
\begin{tabular}{lcr}
(a)
\begin{tikzpicture}
\draw [dashed] (-1,0)--(1,0);
\draw [dashed] (0,-1)--(0,1);
\draw [thick, ->] (0,1) -- (0,0) -- (1,0);
\draw [semithick, BrickRed] (-0.2,0.2) -- (0.2,-0.2);
\end{tikzpicture}
&\quad \quad\quad\quad\quad\quad
&
(b)
\begin{tikzpicture}
\draw [dashed] (-1,0)--(1,0);
\draw [dashed] (0,-1)--(0,1);
\draw [thick, ->] (0,-1) -- (0,0) -- (1,0);
\draw [semithick, BrickRed] (-0.2,-0.2) -- (0.2,0.2);
\end{tikzpicture}
\end{tabular}
\caption{An illustration of coupling rule (4).  The driving walk (a) and the driven walk (b) are initially traveling in opposite directions.  The driving walk then encounters a mirror at some time $t\in\mcal{T}$, and the driven walk encounters an appropriately modified mirror such that the walks exit the mirror traveling in the same direction.}
\label{fig:mirror-coupling}
\end{figure}
As we shall argue below, the walk $\Phi(\tilde{W})$ has the law of the Lorentz mirror walk starting from $X(0)$ in direction $V(0)$.  This coupling, which corresponds to that of \cite{MR4156218} for the Lorentz gas, is chosen to maximize the set of times for which $V(t)=\tilde{V}(t)$ (rather than, for example, to directly minimize $|X(t)-\tilde{X}(t)|$).  More precisely, this coupling has the property that for all $t\in\mcal{T}$ we have $V(t)=\tilde{V}(t)$ unless this is impossible.  There are two reasons this might be impossible: either $X(t)=X(s)$ for some $s<t$ so that the direction $V(t)$ is forced by the history, or $\tilde{V}(t)=-V(t-1)$ which would force the walk to backtrack.
\medskip

To see that $\Phi(\tilde{W})$ has the law of the Lorentz walk, the only point to check is that the distribution of a new mirror $m(t)$ is uniform in $\mbf{M}_d$, given that $\tilde m(t)$ is itself uniformly chosen in $\mbf{M}_d$. If $V(t-1)=\tilde V(t-1)$, there is nothing to show. Consider the alternative case when $V(t-1) = e \ne \tilde e= \tilde V(t-1)$, and set
\footnote{Note that $\mbf{M}_d(e,\tilde e)=\mbf{M}_d(\tilde e,e)$ because of reversibility
in form of $m(-m(v))=-v$.}
\[
\mbf{M}_d(e,\tilde e):=\{m \in \mbf{M}_d \ |\ m [\tilde e] \ne -e\}.
\]
If $\tilde m(t) \notin \mbf{M}_d(e,\tilde e)$, then we are in case (3)
and $m(t)=\tilde m(t)$.
Assume now that $\tilde m(t)$ is uniformly chosen in $\mbf{M}_d(e,\tilde e)$. We follow the rule of case (4), which takes the form $m(t)=\mathcal{M}_{e,\tilde e}\big( \tilde m(t)\big)$, where $\mathcal{M}_{e,\tilde e}$ is a bijection on $\mbf{M}_d(e,\tilde e)$ since
$m \in \mbf{M}_d(e,\tilde e) \implies  m[ e] \ne -\tilde e$. Hence $m(t)$ is also uniformly chosen in $\mbf{M}_d(e,\tilde e)$, which concludes the argument.

\begin{remark}\label{rem:trivial-time}
If we specialize to the case $X(0)=\tilde{X}(0)$ and $V(0)=\tilde{V}(0)$, then we may define the stopping time of interacting with the past
\begin{equation}
\label{tau-int}
    \tau _{\rm int}:= \min \big\{ t\ge 1 : \exists s\in [0,t) \text{ such that }X(s)=X(t) \text{ and either }s\in \{0\}\cup \mathcal T \text{ or }t\in \mathcal T \big\}.
\end{equation}
Note that for any $t<\tau _{\rm int}$ we have that $X(t)=\tilde X(t)$ and $V(t)=\tilde{V}(t)$. Indeed, otherwise we may consider the first time $t<\tau _{\rm int}$ for which $V(t)\neq \tilde{V}(t)$. Since $V(t-1)=\tilde{V}(t-1)$, at time $t$ we cannot be in cases (3) and (4) above. Since $V(t)\neq \tilde{V}(t)$ we cannot be in case (2) and so we are in case (1). Hence, there is some $s\in [0,t)$ such that $X(s)=X(t)$. Moreover, observe that either $W$ or $\tilde{W}$ turned at time $t$ since otherwise $V(t)=V(t-1)=\tilde{V}(t-1)= \tilde{V}(t)$. Thus, it follows that either $t\in \mathcal T$ or $s':=\min \{r\in [0,t): X(r)=X(t)\}\in \mathcal T$  contradicting the fact that $t<\tau _{\rm int}$. This shows that $V(t)=\tilde{V}(t)$ for any $t<\tau _{\rm int}$. Let us also note that the mirror walk $W$ cannot close before time $\tau _{\rm int}$.
\end{remark}

Likewise, if we now take $\tilde{V}(0)$ to be random and uniformly distributed in $\{\pm e_j\}_{j=1}^d$, $\Phi(\tld{W})$ has the law of a mirror walk starting with a uniform random velocity.

\medskip

In what follows $\mathcal F_{t}$ denotes the $\sigma$-algebra generated by  $\{\tilde W(s)\}_{s\le t}$.
In the construction above it is crucial that on the event that $X(t) \notin \{X(s):0\le s<t\}$ and $\tilde{V}(t-1)\neq V(t-1)$, which is measurable in $\mathcal F _{t-1}$, we have
\begin{equation}\label{eq:prob of coupling}
    \mathbb P \big( V(t)=\tilde{V}(t) \mid \mathcal F_{t-1} \big) \ge \tfrac{2d-2}{2d-1}>0.
\end{equation}

\medskip

If the mirror walk returns to the origin at some time $t$ and has the same speed $V(t)=V(0)$ as in time $0$, the mirror walk closes a loop and repeats its past. Let $\tau _{\rm clo}$ be the first time this happens, that is,
\begin{equation}
\label{tau-clo}
    \tau _{\rm clo}:= \min \{t>0: X(t)=0 \text{ and } V(t)=V(0) \} \in \N \cup \{+\infty\}.
\end{equation}
As claimed above, $\tau _{\rm clo} \ge \tau _{\rm int}$.

\subsection{Regenerated mirror walk}\label{sec:regenerated}
We define the regenerated mirror walk $W(t)$ as follows. Let $\{W^{i}_L(s)\}_{i\geq 1}$ be i.i.d.~Lorentz mirror walks starting from the $e_1$ direction with closing times $\{\tauclo^{i}\}_{i\geq 1}$. For $i\ge 0$, define $\taureg ^i:=\sum _{j=1}^i \tauclo ^{j}$ and the walk
\begin{equation}\label{eq:regen.walk.defn}
W(s):= W^{i}_L\big( s - \taureg ^{i-1} \big) \quad \hbox{for } s\in \big[ \taureg ^{i-1},\taureg ^i \big] .
\end{equation}
In words, $W(t)$ sequentially follows the mirror walks $W^{i}_L(s)$ up until they close a loop and then moves onto the next one. We say that $\taureg ^i$ the $i$-th regeneration time of the regenerated walk $W$. Let $\alpha (t)$ be the last regeneration before time $t$.

\medskip

Just like for the mirror walk, one can drive a regenerated mirror walk using a driving walk. Indeed, we simply drive the $i$-th mirror walk $\{W^i_L(s)\}_{s>0}$ in the definition of the regenerated walk using the driving walk $\{\tilde{W}(\tau _{\rm reg}^{i-1}+s)\}_{s\ge 0}$ (in here we also shift the set $\mathcal T$).

\medskip

We can set an initial position $x_0$ and velocity $v_0$ for the regenerated mirror walk by translating and rotating $W$.  From now on, $W$ always denotes the regenerated mirror walk with a chosen initial velocity (being $e_1$ unless otherwise specified), and $W^*$ the regenerated mirror walk with a uniform random initial velocity.

\subsection{Preliminary estimates}
\label{sec:prelims}
We consider a regenerated mirror walk $W$ driven by $\tld{W}$ as described in the previous section. Let $\mcal{M}(t)$ be the set of mirror locations
\begin{equation}
\label{eq:M-def}
    \mathcal M (t):=\{0\} \cup \{X(s):s\in \mathcal T \cap [\alpha (t),t] \} \subset \mathbb{Z}^d.
\end{equation}
This set contains the locations of the mirrors discovered in the time range $[\alpha(t),t]$ by the regenerated walk (that is, since its last regeneration time).\footnote{To be precise, this set is a super set of the locations of mirrors of the regenerated walk. Indeed, if the walk crosses itself at a point which had no mirror, one cannot add a mirror at this later time $t'$ even if $t'\in \mathcal T$.}
Note that $X(t')\in\mathcal M(t)$ does not imply $t'\in\mcal{T}$ (because the walk can revisit a previously discovered mirror).  Also note that $t\in\mcal{T}$ does not imply $V(t+1)\not=V(t)$ (in fact, the discovered mirror could even be the identity).

\medskip

Set $t_*:=\lfloor p^{-1}\log ^3 (1/p) \rfloor $, time scale we shall use throughout this contribution (at this time scale, the walk starts being diffusive).

\medskip

Discovering too many (resp. few) mirrors in some short time interval is likely to slow down (resp. accelerate) the walk too much, which we want to avoid. We thus define the stopping time of discovering too many new mirrors in a short time interval by
\begin{equation*}
\tau _{\rm many}:= \min \big\{ t\ge p^{-1}: |\mathcal T\cap [t-p^{-1},t]|\ge \log ^3(1/p)\big\} ,
\end{equation*}
and the stopping time of walking in a straight line for too long
\begin{equation*}
\tau _{\rm few}:= \min \big\{t\ge t_*: V(t) = V(t') \text{ for all } t'\in[t-t_*,t]\big\}.
\end{equation*}
These stopping times are controlled as follows.

\begin{lemma}
\label{lem:fewhea}
    We have that
    \begin{equation}
        \mathbb P \big( \tau _{\rm many} \wedge \tau _{\rm few}  \le e^{\log ^2(1/p)} \big) \le e^{-c\log ^3(1/p)}.
    \end{equation}
\end{lemma}

\begin{proof}
Let $\bar{T}:=\lfloor e^{\log ^2(1/p)} \rfloor $ and consider first the stopping time $\tau _{\rm many}$. The random variables $\mathds 1\{t\in \mathcal T\}$ for $t\in \mathbb N$ are just independent Bernoulli variables with success probability $p$ and therefore for any fixed $t$ we have $\mathbb P \big( \mathcal T \cap [t-p^{-1},t] \ge \log ^3(1/p) \big) \le e^{-c\log ^3(1/p)}$ and the estimate $\mathbb P ( \tau _{\rm many} \le \bar T ) \le e^{-c\log ^3(1/p)}$ follows from a union bound.

We turn to the proof of the corresponding bound for $\tau _{\rm few}$. To this end, recall the construction of the regenerated walk $W$ using the independent mirror walks $W^i_L$. If $\tau _{\rm few} \le \bar T$ then there is $i\le \bar T$ and $t\le \bar T$ for which $V^i_L(s)=V^i_L(t)$ for all $s\in [t-\tfrac{1}{2}t_*,t]$, where $V^i_L$ is the speed process of $W^i_L$. If this happens then in the $i$-th mirror environment there are a vertex $x\in \mathbb Z ^d$ with $\|x\|\le \bar T$ and a direction $v_0$ such that the light ray starting from $x$ in direction $v_0$ does not change direction for time $\tfrac{1}{2}t_*$. This happens with probability $e^{-c\log ^3(1/p)}$ for a specific $i\le T$, a vertex $x\in \mathbb Z ^d$ with $\|x\|\le \bar T$ and a direction $v_0\in \{\pm e_1,\dots ,\pm e_d\}$, and therefore the result follows from a union bound.
\end{proof}

\section{Induction hypotheses and structure of the proof}
\label{sec:hypotheses}
In this section we introduce our inductive assumptions, prove Theorem~\ref{th:diffusive} using these assumptions and explain the structure of the inductive proof.

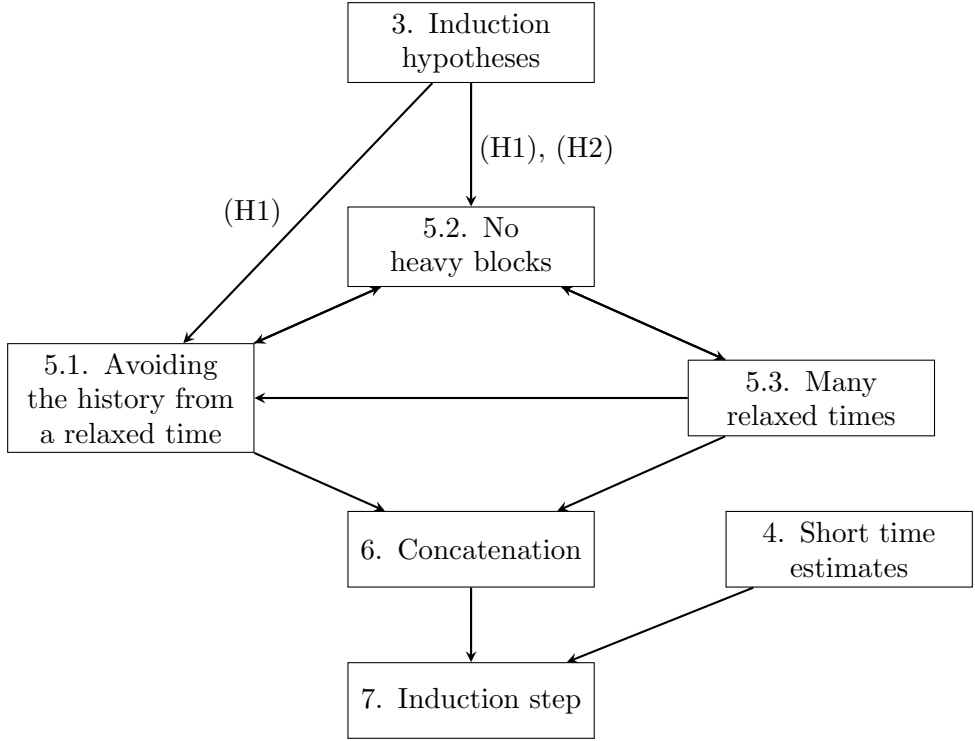
\begin{figure}[t]
\begin{center}
\begin{tikzpicture}[node distance=2cm]

    \node (H4) [process] {3. Induction hypotheses};
    \node (H52) [process, below of=H4,yshift=-0.7cm] {5.2. No heavy blocks};
    \node (H51) [process, left of=H52, xshift=-2.5cm, yshift=-2cm] {5.1. Avoiding the history from a relaxed time};
    \node (H53) [process, right of=H52, xshift=2.5cm, yshift=-2cm] {5.3. Many relaxed times};
    \node (H6) [process, below of=H52, yshift=-2cm] {6. Concatenation};
    \node (H7) [process, below of=H6] {7. Induction step};
    \node (H3) [process, right of=H6, xshift=3cm] {4. Short time estimates};
    \draw [arrow] (H51) -- (H52);
    \draw [arrow] (H52) -- (H51);
    \draw [arrow] (H4) -- node[left,xshift=-0.2cm] {(H1)} (H51);
    \draw [arrow] (H4) -- node[above, xshift=1cm, yshift=-0.4cm] {(H1), (H2)} (H52);
    \draw [arrow] (H52) -- (H53);
    \draw [arrow] (H51) -- (H6);
    \draw [arrow] (H53) -- (H6);
    \draw [arrow] (H53) -- (H52);
    \draw [arrow] (H6) -- (H7);
    \draw [arrow] (H3) -- (H7);
    \draw [arrow] (H53) -- (H51);
\end{tikzpicture}
\end{center}
\caption{The structure of the argument.  Section 5 sets up the estimates necessary to perform the concatenation argument in Section 6.  In Section 7 the proof is completed using the short-time estimates from Section 4 with the concatenation argument to perform the induction.}
\label{fig:arg-structure}
\end{figure}

\subsection{Induction hypotheses and main result}
Recall that $W=(X,V,m)$ is the {\bf regenerated} mirror walk starting from $0$ with $V(0)=e_1$. The inductive assumption at time $T$ is twofold:
\begin{h1}
\label{h1}
Fix $\epsilon :=0.01$. For any $z\in \mathbb Z ^d$ we have\footnote{When using Assumption~$\Hypone{T}$, we will often bound the right hand side of \eqref{eq:h1bound} by $(pT)^{-d/2} p^{-2\epsilon}$ using that $T\le e^{\log ^2(1/p)}$ and that $p$ is sufficiently small.}
\begin{equation}\label{eq:h1bound}
    \mathbb P \big( \|X(T)-z\| \le p^{-1} \big) \le (pT)^{-d/2} p^{-\epsilon }e^{(\log T)^{1/3}}.
\end{equation}
\end{h1}

\begin{h2}
\label{h2}
For any $s\le t\le T$ we have that
\begin{equation*}
    \mathbb P \Big( \|X(t)-X(s)\| \ge  \log ^8(1/p) \sqrt{(t-s)/p}  \Big) \le e^{- 2\log ^2p}.
\end{equation*}
\end{h2}
We write $\Hypone{t}$ to mean the statement that \hyperref[h1]{\bf{(H1)}} holds at time $t$, and likewise with $\Hyptwo{t}$.  Moreover,
we write $\Hyp(T)$ to mean the statement that
$\Hypone{t}$ and $\Hyptwo{t}$ holds for all $0\leq t\leq T$.
\begin{remark}
\label{rem:trivial}
Assumption $\Hypone{T}$ trivially holds when $T\le p^{-1-2\epsilon/d}$ while Assumption~$\Hyptwo{T}$ trivially holds for $T< p^{-1} \log^{16} (1/p)$.
\end{remark}

Our main result is the validity of the induction hypotheses for long times.
\begin{thm}\label{th:main}
    Let $\bar T:= \lfloor e^{\log ^2(1/p)} \rfloor$. For all $p$ sufficiently small, $\Hyp(\bar T)$
    holds, and moreover for all $p^{-5/4} \le T\le \bar T$ we have
    \begin{equation}\label{e.th:main-var}
    \Big|\frac p T\,\E \big[ \|X(T)\|_2^2\big] -\frac{2d-1}{d-1}
    \Big| \le  \tfrac12 p^{1/9}.
    \end{equation}
\end{thm}

\subsection{Proof of Theorem~\ref{th:diffusive}}

As an immediate consequence of Theorem~\ref{th:main}, we obtain that the trajectory of a non-regenerated mirror walk does not close with positive probability up to time $\bar T:=\lfloor e^{\log ^2(1/p)} \rfloor $.
\begin{corollary}\label{cor:clo}
    We have that $\mathbb P \big( \tauclo \le e^{\log ^2 (1/p)} \big) \le p^{1/3}$.
\end{corollary}
\begin{proof}[Proof of Corollary~\ref{cor:clo}]
    First we claim that $\mathbb P (\tauclo \le p^{-3/2}) \le 3\sqrt{p}$. Indeed, in order to close before time $p^{-3/2}$, we have to cross the ray $\{-s \cdot m(0)[-e_1]:s\ge 0\}$ and find a new mirror to turn and enter the ray. The number of times the walk enters this ray up to time $t$ is at most the number of mirrors discovered by this time, and every time it crosses, it has probability $p$ to find a new mirror (and possibly turn). Thus
    \begin{equation*}
     \mathbb P (\tauclo \le p^{-3/2}) \le p \cdot (2p^{-1/2}) +\mathbb P \big( |\mathcal M (p^{-3/2})|\ge 2p^{-1/2} \big) \le 3\sqrt{p},
    \end{equation*}
    where we used that the number of mirrors discovered is bounded by $|\mathcal T \cap [0,t]|$.
    Next, before the walk closes it is identical to the regenerated walk and we can use the transition probability bounds from Assumption~\hyperref[h1]{\bf{(H1)}}. Letting $j_0:=\lfloor p^{-1/2}  \rfloor $ and $j_1:=\lfloor e^{\log ^2(1/p)} \rfloor$ we obtain, using $d\ge 4$,
        \[\mathbb P \big( p^{-3/2} \le \tauclo \le e^{\log ^2(1/p)} \big) \le \sum _{j=j_0}^{j_1} \mathbb P \big( \|X(j \lfloor p^{-1} \rfloor )\|\le p^{-1} \big) \le p^{-2\epsilon }\sum _{j=j_0}^{j_1} j^{-d/2}\le p^{1/3}.\qedhere \]
\end{proof}
We may now proceed with the proof of Theorem~\ref{th:diffusive}.
On the event $\tauclo>T$ we have $X_L(T)=X(T)$, and by Cauchy-Schwarz' inequality,
\begin{eqnarray*}
{\Big|\frac pT \mathbb E[\| X(T)\|_2^2]- \frac pT \mathbb E[\|X_L(T)\|_2^2] \Big|}
&\le& \frac pT\, \mathbb E\Big[ (\|X(T)\|_2^2+\|X_L(T)\|_2^2) \mathds1_{\tauclo<T} \Big]
\\
&\le &2\, \frac pT \,\mathbb E\Big[ \max_{0\le t \le T}\| X(t)-X(0)\|_2^4 \Big] ^\frac12 \mathbb P(\tauclo<T)^\frac12 \,\le \, \tfrac12 p^{1/9},
\end{eqnarray*}
where in the second inequality we also used that $\|X_L(T)\|_2\le \max _{t\le T} \|X(t)\|_2$ if $\tau _{\rm clo}<T$ and $X(0)=0$, and in the last inequality we used Assumption~\hyperref[h2]{\bf{(H2)}} for the first factor and Corollary~\ref{cor:clo} for the second factor. This entails the variance estimate of Theorem~\ref{th:diffusive} in combination with \eqref{e.th:main-var}.

\subsection{Structure of the proof of Theorem~\ref{th:main}}

To complete the induction argument we need a short-time estimate and an inductive step.  The short-time estimates are completed in Section~\ref{sec:KP} using an adaptation of the Kesten-Papanicolaou argument, which amount to an exact coupling of $\tilde W$ and $W$.   The main engine behind the inductive step is a concatenation lemma, which states that the mirror walk up to time $t_1+t_2$ can be coupled to be close to the concatenation of two independent mirror walks of lengths $t_1$ and $t_2$ (see Proposition~\ref{prop:conc}).  For this concatenation step to work we need to quantify the probability that a walk interacts with its previous history.  This motivates the notion of a \textit{relaxed time} (see Definition~\ref{eq:def_relax} below).
To set up the concatenation argument we therefore need to show that, indeed, the walk is unlikely to interact with its history from a relaxed time (Proposition~\ref{prp:dont-hit}), and that here are many relaxed times (Proposition~\ref{prp:relaxed}). In Section 5 we prove that these hold, assuming the induction hypotheses.  This sets up the concatenation argument, which then allows the induction argument to be completed. Throughout Sections~\ref{sec:heavy}, \ref{sec:conc} and \ref{sec:induct} we assume $\mathcal H(T)$ and in Section~\ref{sec:induct} we finish the induction by proving $\mathcal H(T+1)$.

The structure of the proof is depicted in Figure~\ref{fig:arg-structure}.

\section{The Kesten-Papanicolaou argument}\label{sec:KP}
In this section we introduce an argument that will be reused in the
induction scheme and that can be implemented to give a short proof of a coupling between the driven walk and the mirror walk up to times   $p^{-5/4}$ which is already diffusive\footnote{A more careful argument would yield a coupling to the driven walk up to times $\ll p^{-2}$.}.

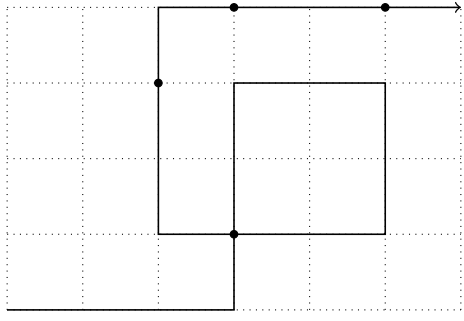
\begin{figure}
\centering
\begin{tikzpicture}
\draw [dotted] (-3,-3) -- (3,-3);
\draw [dotted] (-3,-2) -- (3,-2);
\draw [dotted] (-3,-1) -- (3,-1);
\draw [dotted] (-3,0) -- (3,0);
\draw [dotted] (-3,1) -- (3,1);

\draw [dotted] (-3,-3) -- (-3,1);
\draw [dotted] (-2,-3) -- (-2,1);
\draw [dotted] (-1,-3) -- (-1,1);
\draw [dotted] (0,-3) -- (0,1);
\draw [dotted] (1,-3) -- (1,1);
\draw [dotted] (2,-3) -- (2,1);
\draw [dotted] (3,-3) -- (3,1);

\draw [semithick, ->] (-3,-3) -- (0,-3) -- (0,0) -- (2,0) -- (2,-2) -- (-1,-2) -- (-1, 1) -- (3,1);

\filldraw [black] (0,1) circle (0.05);
\filldraw [black] (2,1) circle (0.05);
\filldraw [black] (-1,0) circle (0.05);
\filldraw [black] (0,-2) circle (0.05);
\end{tikzpicture}
\caption{An illustration of the set $\Ss$ defined in~\eqref{eq:S}.  At the times marked by circles, the path could turn in some direction to quickly hit a previously discovered mirror.  So long as $\mcal{T}\cap \mcal{S} = \emptyset$, the walk is guaranteed to explore fresh environments.}
\label{fig:Ss}
\end{figure}

Recall that $W$ is a regenerated mirror walk starting from $X(0)=0$ in direction $V(0)=e_1$. Our goal is to show that $W$ rarely interacts with its past, so that it can remain coupled to the driving walk $\tilde{W}$ (see Section~\ref{sec:driving}). To prove this, it is helpful to introduce the set of times $\Ss$ at which $W$ may turn in the direction of an old mirror. Define
\begin{equation}\label{eq:S}
    \Ss :=  \bigg\{ s>0 : \begin{array}{cc}
         \text{ there exist a mirror at } y\in \mathcal M(s-1) \text{ and a direction }  \\
         u \neq \pm V(s-1)\text{ such that } y+ru=X(s) \text{ for some } r\in [0,2t_*]
    \end{array} \bigg\},
\end{equation}
where we recall that $t_*=\lfloor p^{-1}\log ^3 (1/p) \rfloor $.
See Figure~\ref{fig:Ss} for an illustration.
The choice $r\in [0,2t_*]$ is natural since the walk cannot go straight for longer than $t_*$ if $s+2 t_*\le \tau _{\rm few}$.
Note that the event $\{s\in \Ss\}$ is measurable in $\mathcal F_{s-1}$.

The main combinatorial lemma which allows us to control self-interactions is the following, which we will reuse several times also in the proof of the long-time result.
\begin{lemma}[Sparsity of $\Ss$]
\label{lem:sparsity}
For any $a<b$, we have on the event $\tau _{\rm few} >a$
\[
|\Ss\cap[a,b]| \leq 2(|\mcal{M}(a)\cap B_{b-a+2t_*}(X(a))| + |\mcal{T}\cap[a,b]|)^2.
\]
\end{lemma}
\begin{proof}
We say that a mirror at  $y\in \mathcal M(b)$ is ``relevant'' if it can cause $s \in \Ss \cap [a,b]$. By definition of $\Ss$, this means
that either $y$ belongs to $\mathcal M(a)$ and is at distance $2t_*$ of the trajectory of $X$ on $[a,b]$ (that is, $y\in B_{b-a+2t_*}(X(a))$), or the mirror at $y$ is discovered in the time interval $[a,b]$ (that is, $y=X(s')$ for some $s'\in \mcal{T} \cap[a,b]$). We thus define
\[
A = \{X(s) \,:\, s\in\mcal{T}\cap[a,b]\} \cup (\mcal{M}(a)\cap B_{b-a+2t_*}(X(a)))
\]
Note that $|A|\leq |\mcal{M}(a)\cap B_{b-a+2t_*}(X(a))| + |\mcal{T}\cap[a,b]|$.
Consider the function $f:\Ss\cap[a,b] \to A\times A$ defined by
\[
f(s) = (y_1(s),y_2(s)) ,
\]
where $y_1(s)$ is the last mirror visited\footnote{The term ``mirror'' is slightly abusive since $A$ is only a super set of the mirror locations -- it is only meant here as a potential mirror location.} by the walk before time $s$ (because $\tau _{\rm few} >a$, $y_1(s) \in A$ -- recall that $0\in \mcal{M}(a)$) and $y_2(s)$ is the mirror such that $y_2(s)+ru=X(s)$ in the definition of $\Ss$ (likewise, $y_2(s)\in B_{b-a+2t_*}(X(a))$). This function is thus well-defined.

Given a point $y\in \Z^d$ define the set $\star(y)$ to be
\[
\star(y) := \{y+r v \,:\, v\in \{\pm e_j\}_{j=1}^d, r\in\N \}.
\]
Then by the definition of $y_1(s)$ and $y_2(s)$ it holds that
$X(s) \in \star(y_1(s)) \cap \star(y_2(s))$.  Since
for any $y,y'\in\Z^d$ we have $|\star(y)\cap\star(y')| \leq 2$, the map
$f$ is $2$-injective, and $|\Ss\cap[a,b]|\leq 2|A|^2$, as claimed.
\end{proof}

Using this combinatorial lemma we establish the coupling between the regenerated walk and the driving walk.

\begin{proposition}\label{prop:coup-for-variance}
    Let $W$ be the regenerated mirror walk driven by $\tilde{W}$, with both walks starting from $X(0)=\tilde{X}(0)=0$ in direction $V(0)=\tilde{V}(0)=e_1$. For all $T\ge 1$ we have
    \begin{equation}
    \label{eq:KP-coupling}
        \P \big( \forall t\le T, \ X(t)=\tilde{X}(t) \big) \ge 1-p^3 T^2 \log^6(1/p).
    \end{equation}
\end{proposition}

\begin{proof}
First, we claim that on the event $\Omega :=\{\mcal{T}\cap\mcal{S}\cap[0,T] =\emptyset \} \cap \{\tau _{\rm few} >T\}$, we have that $X(t)=\tilde{X}(t)$ for all $t\le T$.
To this end, it suffices to show that we have $\tau _{\rm int}\ge T$ on this event, where $\tau _{\rm int}$ is defined in Remark~\ref{rem:trivial-time}. This will also show that $X$ does not regenerate up to time $T$ and therefore the discussion in Remark~\ref{rem:trivial-time} about non-regenerated mirror walks holds in this case. By the definition of $\tau _{\rm int}$ we have either that $X(\tau _{\rm int})\in \mathcal M(\tau _{\rm int}-1)$ or that $X(\tau _{\rm int}) =X(s)$ for some $s\in [0,t)$ and $\tau _{\rm int}\in \mathcal T$. Let us show that if $\tau _{\rm int}\le T$ then none of these cases can occur on the event $\Omega $, which thus implies that $\tau_ {\rm int} >T$. Suppose first that $\tau _{\rm int}\in \mathcal T$ and that $X(\tau _{\rm int}) \notin \mathcal M(\tau _{\rm int}-1)$ (that is, we hit a straight line at $\tau _{\rm int}$). We have that $X(\tau _{\rm int})=X(s)$ for some $s\le\tau _{\rm int}$ and by our assumption $V(s)\neq \pm V(\tau _{\rm int}-1)$ since the walk can only close at the origin, and $X(\tau_{\rm int}) \notin \mathcal M(\tau_{\rm int}-1)$ implies
$X(\tau_{\rm int})\ne 0$. Since $\tau _{\rm few}> T\ge s$, there exists $r\in [0,t_*]$ such that $X(s)-rV(s)\in \mathcal M(\tau _{\rm int})$ and therefore $\tau _{\rm int} \in \mathcal S \cap \mathcal T \cap [0,T]$, which is impossible on $\Omega $. Next, suppose that $X(\tau _{\rm int}) \in \mathcal M(\tau _{\rm int}-1)$ (that is, we hit a previously discovered mirror). Let $s\in [0,\tau _{\rm int})$ be the last turn of the walk before $\tau _{\rm int}$ and observe that $\tau _{\rm int}-s\le t_*$ since $\tau _{\rm few}> T$. Moreover, $s$ must be the first visit to the mirror in $X(s)$ since otherwise $X(s)\in \mathcal M(s-1)$, contradicting the definition of the stopping time $\tau _{\rm int }$. It follows that $s\in \mathcal T$ and therefore $s\in \mathcal S \cap \mathcal T \cap [0,T]$ ($s\in \mathcal S$ using that $y=X(\tau_{\rm int})$), which is impossible on $\Omega$.

We now estimate the probability of $\Omega$.
Set $\Ss\cap[0,T]:=\{s_1,\dots,s_K\}$ with $s_1< s_2< \dots <s_K$ and $K=|\Ss\cap[0,T]|\le T+1$.
We then have
\[
\Prob(\mcal{T}\cap \mcal{S}\cap[0,T] \not=\emptyset ) \,\leq\, \sum_{i=1}^{T+1} \Prob\big(K\ge i,\  s_i \in \mcal{T})
\,= \, \sum_{i=1}^{T+1}\sum_{s\le T} \Prob\big(  K\ge i,\  s_i=s, \ s \in \mcal{T}\big).
\]
Since $\{K\ge i,\ s_i=s\}$ is measurable wrt $\mathcal F_{s-1}$, it is independent of $\{s \in \mcal{T}\}$ and so,
\[\Prob\big( K\ge i,\  s_i=s, \  s \in \mcal{T}\big)
= p \cdot \Prob\big( K\ge i,\
 s_i=s\big),\]
so that
\[
\Prob(\mcal{T}\cap \mcal{S}\cap[0,T] \not=\emptyset ) \,\leq\, p\sum_{i=1}^{T+1}\sum_{s\le T}  \Prob\big( K\ge i,\ s_i=s\big)
\,= \, p\sum_{i=1}^{T+1} \Prob\big(K\ge i) .
\]
We then estimate the size of $K$ using sparsity.
By Lemma~\ref{lem:sparsity}, which for $[a,b]=[0,T]$ takes the simpler form of $|\Ss\cap[0,T]| \leq 2 (1+|\mcal{T}\cap[0,T]|)^2$ (recall that $\tau _{\rm few}>0$ and $\mcal{M}(0)=\{0\}$),
and by Lemma~\ref{lem:fewhea} (for $\tau_{\rm many}$), $K\le \frac14 p^2 T^2 \log^6 (1/p)$ with very high probability, so that
\[
\Prob(\mcal{T}\cap \mcal{S}\cap[0,T] \not=\emptyset ) \,\le\, \tfrac12 p^3 T^2 \log^6 (1/p).
\]
In combination with Lemma~\ref{lem:fewhea} again (this time for $\tau_{\rm few}$), this implies $\Prob(\Omega^c)\le p^3 T^2 \log^6 (1/p)$, and concludes the proof.
\end{proof}

Next, we estimate the variance of the driving walk $\tilde{W}$, starting from a uniform random direction. Recall that $\tilde{W}$ is simply a non-backtracking random walk that continues straight with probability $1-p$ and turns with probability $p$ to a uniform direction not including where it came from.

\begin{lemma}\label{lem:driving}
    Let $\tilde{W}$ be a driving mirror walk starting from $\tilde{X}(0)=0$ with $\tilde{V}(0)$ being a uniform random direction in $\{\pm e_j\}_{j=1}^d$. Then, for any $t\ge 0$ we have that $\mathbb E\tilde{X}(t)=0$ and $\cov (\tilde{X}(t))=\tilde{\sigma} _t^2 \Id$ with
    \begin{equation*}
        \tilde{\sigma} _t^2:= \frac{1+\rho}{d(1-\rho)} t +\frac{2\rho (1-\rho ^t)}{d(1-\rho )^2} \quad \text{where} \quad \rho:=1-p\frac{2d-2}{2d-1},
    \end{equation*}
which yields
\[
\Big|\frac pt \tilde{\sigma} _t^2-\frac{2d-1}{d(d-1)}\Big| \le C\big(p+\frac{1}{pt}\big).
\]
\end{lemma}

\begin{proof}
    We have that $\mathbb E[\tilde{V}(t)\cdot \tilde{V}(t+1)]=1-p\frac{2d-2}{2d-1}=\rho $ and therefore $\mathbb E [\tilde{V}(t+1) \mid \tilde{V}(t)]=\rho  \tilde{V}(t)$ by symmetry. Thus, by induction we have for all $s<t$
    \begin{equation}\label{eq:812}
        \mathbb E[\tilde{V}(s)\cdot \tilde{V}(t)]= \rho^{t-s} .
    \end{equation}
    Indeed, assume that \eqref{eq:812} holds. Then, since $\tilde{V}(t+1)$ is measurable wrt $(\tilde V(s),\pi(s+1),\dots,\pi(t+1))$ and $\pi$ is iid,
    \begin{equation*}
        \mathbb E[\tilde{V}(s)\cdot \tilde{V}(t+1)]= \mathbb E\big[ \tilde{V}(s) \cdot   \mathbb E [\tilde{V}(t+1) \mid \tilde{V}(t)] \big] =\rho \mathbb E[ \tilde{V}(s) \cdot \tilde{V}(t)] =\rho ^{t+1-s}.
    \end{equation*}
    Thus,
    \begin{equation*}
    \begin{split}
        \mathbb E [\|\tilde{X}(t)\|^2] &= \mathbb E \Big[ \Big( \sum _{s=0}^{t-1}\tilde{V}(s) \Big)\cdot \Big( \sum _{s=0}^{t-1}\tilde{V}(s) \Big) \Big] =t+2\sum _{s=0}^{t-1}\sum _{s'=0}^{s-1} \mathbb E [\tilde{V}(s')\cdot \tilde{V}(s)]\\
        &=t+2\sum _{s=0}^{t-1}\sum _{s'=0}^{s-1}\rho ^{s-s'}= t+2\sum _{s=0}^{t-1} \rho ^s \frac{\rho ^{-s}-1}{\rho ^{-1} -1}= t+\frac{2\rho }{1-\rho }\sum _{s=0}^{t-1} (1-\rho ^s)
        \\&=t+\frac{2\rho }{1-\rho } \Big(t- \frac{1-\rho ^t}{1-\rho }\Big).\qedhere
    \end{split}
    \end{equation*}
\end{proof}

We now consider  the regenerated walk $W^*$  starting with a uniform random velocity. By symmetry, for any $t$ we have that $\mathbb E [X^*(t)]=0$ and the covariance matrix of $X^*(t)$ is scalar, namely, $\cov(X^*(t))=\sigma _t^2 \Id$ for some standard deviation $\sigma _t>0$. We conclude this section by estimating $\sigma _t$ for short times $t$.

\begin{corollary}\label{cor:coup-for-variance}
For all $t\ge 1$,
    \begin{equation}
    \label{eq:variance-bd}
        \frac p t|\sigma _t ^2-\tilde{\sigma }_t^2| \le 2 p^4 t^3\log ^6(1/p),
    \end{equation}
and in particular, for $1\le t\le p^{-5/4}$,
\[
\frac p t|\sigma _t ^2-\tilde{\sigma }_t^2| \le p^{1/5}.
\]
\end{corollary}

\begin{proof}
Let $W^*=(X^*,V^*)$ and $\tilde{W}=(\tilde{X},\tilde{V})$ be the regenerated mirror walk and the driving walk, both starting from $0$ and the same uniform random direction. Proposition~\ref{prop:coup-for-variance} is stated for these walks starting from the $e_1$ direction, however, we can clearly apply the proposition starting from the same uniform random direction. The proposition then gives a coupling of $W^*$ and $\tilde{W}$ such that $\mathbb P (\exists s\le t, \ X^*(s)\neq \tilde{X}(s)) \le p^3t^2 \log ^6(1/p)$.
Letting $X^*(t)_1$ and $\tilde{X}(t)_1$ denote the first coordinate of $X^*(t)$ and $\tilde{X}(t)$ we obtain
\begin{equation*}
|\sigma _t^2 -\tilde{\sigma _t}^2| =\big| \E [X^*(t)_1^2-\tilde{X}(t)_1^2] \big|\le 2t^2\mathbb P (X^*(t)\neq \tilde{X}(t))\le   2p^3 t^4\log^6(1/p).
\end{equation*}
\end{proof}

\section{Self-interactions, heavy blocks, and relaxed times}\label{sec:heavy}

The fundamental idea, even present in the Kesten-Papanicolaou argument
above, is that the random walk behaves like a Markov walk so long as it does not interact with its past.   To make this idea precise we introduce the notion of hitting time $\tau_{\rm hit}(t)$, which is the first
time that the walk after time $t$ interacts with the walk before time $t$:
\begin{equation}\label{eq:def_tau_hit}
    \tau _{\rm hit}(t):= \bigg\{ t'>t : \begin{array}{cc}
         \text{ there exists } s\in[\alpha(t'),t] \text{ such that }X(t') = X(s)  \\
\text{ and either } t'\in \mathcal T \text{ or } s\in \mcal{T} \cup \{\alpha (t')\}
    \end{array} \bigg\},
\end{equation}
That is, $\tau_{\rm hit}(t)$ is the first time at which either the walk recollides at a previously visited mirror, or at which the walk revisits a straight segment of its past trajectory and is prevented from placing a new mirror, see Figure~\ref{fig:twocases}.

\medskip

The point of $\tau_{\rm hit}(t)$ is that it is the time up to which $W$ can be coupled with an independent mirror walk
starting at $X(t)$ in direction $V(t)$:

\begin{claim}\label{claim:cou}
Let $\{W(s)\}_{s\ge 0}$ be the regenerated mirror walk driven by $\{\tilde{W}(s)\}_{s\ge 0}$. Let $t\ge 0$ and let $W'=(X',V')$ be the regenerated walk starting at $X'(0)=X(t)$ and $V'(0)=V(t)$ and driven by $\{\tilde{W}(t+s)\}_{s\ge 0}$. Then, we have that $X(t+s)=X'(s)$ and $V(t+s)=V'(s)$ for all
$s\leq \tau_{\rm hit}(t)-t$.
\end{claim}

If it held for example that $\tau_{\rm hit}(t) > T$ for all $t\le T$, then this would imply a Markov property for the walk up to time $T$. This is of course too strong of a statement to hope for, and there are essentially two obstacles that can cause $\tau_{\rm hit}(t)$ to occur shortly after $t$.

\medskip

The first obstacle is a global one: if there are \textit{heavy blocks}, meaning balls that have many more than the expected number of discovered mirrors, then it is likely that the walk will revisit one such mirror.
Recall that $\mathcal{M}(t)$ is the set of mirrors discovered up to time $t$ (since the last regeneration time).
\begin{defn}
    For $x\in \mathbb Z ^d$ and $r\ge p^{-1}$ we say that the block $B(x,r)$ is heavy at time $t$ if
    \begin{equation}\label{e.heavy}
        \big| B(x,r) \cap \mathcal{M} (t) \big|\ge  p^{1.9} r^{2}.
    \end{equation}
\end{defn}
The scaling~\eqref{e.heavy} is natural. If the walk is diffusive with time measured in units of $p^{-1}$ (as we want to establish), then the walk is expected to exit the ball $B(x,r)$ after some time $t$ such that $\sqrt{t/p} \sim r$, that is $t\sim p r^2$. In that time interval, the walk discovers of the order of $p^2 r^2$ mirrors, which is much smaller than the right-hand side $p^{1.9}r^2$ of~\eqref{e.heavy}. If the walk is indeed diffusive, heavy blocks must be unlikely.

\medskip

The second obstacle that could cause $\tau_{\rm hit}(t)$ to occur relatively
quickly after $t$ is a local one, simply that it could hold for some $s\approx p^{-1}$ that $X(t)+sV(t) \in \mathcal{M}(t)$.  This is an obstacle that is characteristic of models that have a short-time
ballistic regime.

\medskip
We now introduce the notion of a \textit{relaxed time}, which is a time for which neither the global obstacle nor the local obstacle forces the walk to quickly interact with its past. Recall that $t_*:=p^{-1}\log ^3(1/p)$.

\begin{defn}\label{eq:def_relax}
A time $t>0$ is called \emph{relaxed} if:
\begin{enumerate}
    \item
    For all $r\ge p^{-1}$ the block $B(X(t),r)$ is not heavy at time $t$.
    \item
    For all $1\le s\le t_*$ we have that $X(t)+sV(t)\notin \mathcal{M}(t)$.
\end{enumerate}
Moreover we say that $t$ is \emph{locally relaxed} if only the second item holds (whether or not there is a heavy block at time $t$).
We write  $\mcal{R}$ for the set of relaxed times.
\end{defn}
 In words, a time is relaxed if the history of the walk is not heavy around $X(t)$ and if $W$ is not about to bump into a mirror.  The key fact about relaxed times is that indeed we can show that
$\tau_{\rm hit}(t)$ is likely to be much larger than $t$ when $t$ is relaxed and when the mirror trajectory is somewhat regular.

To use this idea we need to be able to verify that there are many relaxed times.  Because there are two distinct obstacles
one must avoid, this argument splits naturally into a \textit{global} and a \textit{local} component.  To present the argument it is useful to introduce two new stopping times.   The first stopping time $\tau_{\rm hea}$ is the time at which the first heavy block is formed,
\begin{equation}
\label{eq:tau-hea-def}
\tau_{\rm hea} := \min\{t>0 \,:\, \text{there exists a heavy block at time t}\}.
\end{equation}
The other stopping time we introduce is useful for dealing with the local obstruction, and we
set
\begin{equation}\label{eq:tau-rel}
\tau_{\rm rel} := \min\{t>t_* \,:\, |\mathcal{R}\cap[t-t_*,t]| \leq 0.9t_*\}.
\end{equation}
Finally, we introduce
\begin{equation}\label{eq:deftau}
\tau := \tau_{\rm hea} \wedge \tau_{\rm rel} \wedge \tau_{\rm many} \wedge\tau_{\rm few}.
\end{equation}
The rest of this section is devoted to the proof of the next propositions.

\begin{proposition}[Escaping from a relaxed time]
\label{prp:dont-hit}
If $\Hyp(T)$ holds for some $T \le \bar T$, then on the event  $\{t \text{ is relaxed} \}$ we have
\begin{equation}
\label{eq:orca}
    \mathbb P \big( \tau _{\rm hit}(t) \ge (t+T) \wedge \tau \mid \mathcal F _t  \big) \ge 3/4,
\end{equation}
and, on the same event,
\begin{equation}
\label{eq:humpback}
    \mathbb P \big( \tau _{\rm hit}(t_+) \ge (t+T)\wedge \tau \mid \mathcal F _t  \big) \ge 3/4,
\end{equation}
where $t_+:=\min \{s>t :s\in \mathcal T \}$.
\end{proposition}
\begin{proposition}[Heavy blocks are unlikely]
\label{prp:noheavy}
If $\Hyp(T)$ holds for some $T\le \bar T$, then
\[
\Prob(\tau_{\rm hea} \leq T \wedge \tau ) \leq e^{-c\log^3(1/p)}.
\]
\end{proposition}
\begin{proposition}[Relaxed times are dense]
\label{prp:relaxed}
If $\Hyp(T)$ holds for some $T\le \bar T$, then
\[
\Prob(\tau_{\rm rel} \leq T \wedge \tau ) \leq e^{-c\log^3(1/p)}.
\]
\end{proposition}
On the one hand, the proof of Proposition~\ref{prp:noheavy} uses Proposition~\ref{prp:dont-hit} to show that a walk cannot revisit a given block too many times.
On the other hand, the proof of Proposition~\ref{prp:relaxed} is fundamentally an analysis of the walk at the ballistic timescale (because $\tau_{\rm rel}$ is a stopping time based on the behavior of the walk on such time scales).  It requires an argument that relates the set $\Ss$ defined in~\eqref{eq:S} to the set of relaxed times.

\medskip

Together, Proposition~\ref{prp:noheavy} and Proposition~\ref{prp:relaxed} imply
that $\tau_{\rm rel}$ is large with high probability.
\begin{corollary}
\label{cor:taurel}
If $\Hyp(T)$ holds for some $T\le \bar T$, then
\[
\Prob(\tau \leq T)
\leq e^{-c\log^3 (1/p)}.
\]
\end{corollary}
\begin{proof}
Observe that
\[
\{\tau \leq T \}
\subset \{\tau_{\rm hea} \leq T \wedge \tau  \}
\cup \{\tau_{\rm rel} \leq T \wedge \tau  \}
\cup\{\tau_{\rm many} \leq T \wedge \tau  \}
\cup\{\tau_{\rm few} \leq T \wedge \tau  \}.
\]
The corollary now follows from Proposition~\ref{prp:noheavy}, Proposition~\ref{prp:relaxed} and Lemma~\ref{lem:fewhea}.
\end{proof}

\subsection{Escaping is easy at relaxed times}
\label{sec:donthit}
In this section we prove Proposition~\ref{prp:dont-hit}. Recall the definition of the set $\mathcal S$ from \eqref{eq:S}.

\begin{proof}[Proof of Proposition~\ref{prp:dont-hit}]
We will prove~\eqref{eq:orca}, the proof of~\eqref{eq:humpback} being a small modification.

Recall that $W$ is driven by $\{\tilde{W}(s)\}_{s\ge 0}$ and let $W'=(X',V')$ be a regenerated walk starting from $X'(0)=X(t)$ in direction $V'(0)=V(t)$ that is driven by $\{\tilde{W}(t+s)\}_{s\geq 0}$. Recall also from Claim~\ref{claim:cou} that $X(s)=X'(s-t)$ for all $s\in [t,\tau _{\rm hit}(t)]$.
Next let $j_1:=\lfloor pT \rfloor $. For an integer $j\in [1,j_1]$ let $t_j=(j-1)\lfloor p^{-1} \rfloor $ and let $I_j:=[t+t_j,t+t_j+2t_*]$. Define the events
\begin{equation*}
    \mathcal A _j: = \big\{ d\big( X'(t_j),\mathcal M(t) \big) \le 3t_* \big\} \quad \text{and} \quad \mathcal B _j : = \big\{ [0,\tau] \cap  I_j\cap \mathcal S \cap \mathcal T \neq \emptyset  \big\}.
\end{equation*}

Since $t$ is fixed throughout the proof, we write $\tau_{\rm hit}. =\tau_{\rm hit}(t)$. We split the proof into two steps.
First we prove the inclusion
\begin{equation}
\label{eq:tau-hit-bd}
    \big\{ t \text{ is relaxed}\big\} \cap \big\{ \tau _{\rm hit}\le  (t +T)\wedge \tau \big\} \subseteq  \bigcup _{j=1}^{j_1}\mathcal A _j \cap \mathcal B _j,
\end{equation}
and then we bound the probability of each $\mcal{A}_j\cap \mcal{B}_j$.
Intuitively, this inclusion says that in order to interact with the past in $[0,t]$, we first have to reach a $\approx p^{-1}$ neighborhood of a mirror in $\mathcal M(t)$ and soon after we have to turn exactly at the right time in order to bump into it.

\medskip

\noindent \textit{Step 1.} Proof of~\eqref{eq:tau-hit-bd}.\\
We assume that $t$ is relaxed and that $\tau_{\rm hit} \leq (t+T)\wedge \tau $ and prove that $\mcal{A}_j\cap \mcal{B}_j$ holds for some $1\leq j\leq j_1$.
First we argue that it suffices to show that
\begin{equation}
\label{eq:find-good-time}
[t,\tau ]\cap [\tau_{\rm hit}-t_*,\tau_{\rm hit}] \cap \mcal{S} \cap \mcal{T} \not=\emptyset.
\end{equation}
Indeed, there exists $j\in [1,j_1]$ such that $t+t_j \in  [(\tau_{\rm hit}-2t_*)\vee t,(\tau_{\rm hit}-t_*)\vee t]$, and we claim that $\mcal{A}_j$ and $\mcal{B}_j$ hold for this choice of $j$.  It is clear from~\eqref{eq:find-good-time} and this choice of $t_j$ that $\mcal{B}_j$ holds.
To see that $\mcal{A}_j$ holds, first note that the entire history of the trajectory up to time $t$ is contained in axis-aligned line segments of length $t_*$ centered at each mirror in $\mcal{M}(t)$.  That is, since $t\leq \tau_{\rm few}$  we have
\begin{equation}
\label{eq:history-in-union}
\big\{ X(s) : s\in [\alpha (t),t] \big\} \subset \big\{ x+rv : x\in \mcal{M}(t), v\in \{\pm e_j\}, r \in [0,t_*] \big\}.
\end{equation}
Therefore it follows that
$d(X'(\tau _{\rm hit}(t)-t),\mathcal M(t) )=d(X(\tau_{\rm hit}),\mcal{M}(t)) \leq t_*$, which implies that $d(X'(t_j), \mcal{M}(t)) \leq 3t_*$ and $\mathcal A_j$ holds.

Now we turn to the proof of~\eqref{eq:find-good-time}.
By the definition of $\tau _{\rm hit}$, there must
exist $s^*\in[\alpha(\tau_{\rm hit}),t]$ such that $X(s^*)=X(\tau_{\rm hit})$ and one of the following holds:
either $s^*\in\mcal{T}$ (Figure~\ref{fig:twocases}(a)) or $\tau_{\rm hit}\in\mcal{T}$ and $X(s^*)\notin \mathcal M(t)$ (Figure~\ref{fig:twocases}(b)).

Suppose first that we interact with a straight line segment, so that $\tau_{\rm hit}\in\mcal{T}$ and $X(s^*)\notin \mathcal M(t)$. In this case we claim that in fact $\tau_{\rm hit}\in\Ss$. Indeed, since $t\le \tau _{\rm few}$ we have that $X(s) -rV(s)\in\mcal{M}(t)$ for
some $r\leq t_*$.  Moreover, $V(\tau_{\rm hit})\not=\pm V(s)$ (the mirror walk can only trace a line segment at most once), so this implies that indeed $\tau_{\rm hit}\in \Ss$.

The second case to consider is that we hit a mirror, in which case $X(\tau_{\rm hit}) = X(s^*)\in\mcal{M}(t)$.  Let $\sigma := \max \mcal{R}\cap [t,\tau_{\rm hit}]$ be the last time the walk was relaxed before $\tau_{\rm hit}$.  Because we assume $t$ is relaxed and $\tau_{\rm hit}\le \tau \le  \tau_{\rm rel}$, it follows that $\sigma \in [t,\tau ]\cap [\tau_{\rm hit} - t_*, \tau_{\rm hit}]$.  Now define
\[
\sigma' = \min\{s\geq \sigma : V(s)\not=V(s-1) \}\,.
\]
We will show that $\sigma'\in [t,\tau ]\cap [\tau_{\rm hit}-t_*,\tau_{\rm hit}] \cap \mcal{S} \cap \mcal{T}$.  First, we show  that $\sigma' < \tau_{\rm hit}$. Indeed, otherwise $X(\sigma)+rV(\sigma)=X(\tau _{\rm hit}) = X(s^*)$ for some $r\le t_*$ contradicting the fact that $\sigma$ is relaxed. Now, using that $\sigma'$ is \textit{not} relaxed  and that $\tau _{\rm hea}\ge \tau _{\rm hit}$, it follows that $X(\sigma')+r V(\sigma')\in \mcal{M}(\sigma')$ for some $r\in[1,t_*]$. But now since $V(\sigma')\not= \pm V(\sigma'-1)$, that means that $\sigma'\in\mcal{S}$. Finally, observe that since $\sigma $ is relaxed it follows that at $\sigma '$ the walk discovered a new mirror and so $\sigma '\in \mathcal T$. Hence, $\sigma'\in[(\tau_{\rm hit}-t_*)\wedge t,\tau_{\rm hit}] \cap \mcal{S} \cap \mcal{T}$, as desired.

\medskip

\noindent
\textit{Step 2.} Bounding $\Prob(\mcal{A}_j\cap\mcal{B}_j\mid \mathcal F _t)$.\\
We start by bounding the probability of $\mathcal A _j$. We cover a block of side length $3t_*\le 3p^{-1}\log^3 (1/p)$ by $C \log^{3d} (1/p)$ blocks of side length $p^{-1}$, so that for any $y\in \mathcal M(t)$ we have by
Assumption~\hyperref[h1]{\bf{(H1)}}
\begin{equation*}
    \mathbb P \big( \| X'(t_j)-y \|\le  3t_* \mid \mathcal F _t\big) \le  C\log ^{3d}(1/p) \cdot p^{-2\epsilon }(pt_j)^{-d/2} \le p^{-3\epsilon }j^{-d/2}.
\end{equation*}
Thus, letting $r_j:=\sqrt{t_j/p} \log ^8(1/p) \approx p^{-1}\sqrt{j}\log ^8(1/p)$, and union bounding over the mirrors $y\in M(t)\cap B(X(t),3r_j)$   we obtain
\begin{equation*}
\begin{split}
    \mathbb P \big( \mathcal A _j \mid \mathcal F _t \big) &\le \mathbb P ( X'(t_j) \notin B(X(t),r_j) \mid \mathcal F _t\big) + |\mathcal M(t)\cap B(X(t),2r_j)| p^{-3\epsilon }j^{-d/2} \\
    &\le e^{-c\log ^2 (1/p)} +p^{1.9-3\epsilon }(2r_j)^2 j^{-d/2} \le p^{-0.2}j^{-1},
\end{split}
\end{equation*}
where in the second inequality we used Assumption~\hyperref[h2]{\bf{(H2)}} and that $B(X(t),3r_j)$ is not heavy at time $t$ since $t$ is relaxed. In the last inequality we used that $d\ge 4$.

Next, we bound the probability of $\mathcal B_j$ given $\mathcal A _j$. Since $\mathcal A_j$ is measurable in $\mathcal F _{t+t_j-1}$, it suffices to bound the probability of $\mathcal B _j$ given $\mathcal F _{t+t_j-1}$. Using that $\tau \le \tau _{\rm many}$, we have that $|[0,\tau ] \cap I_j \cap \mathcal T|\le C\log ^6 (1/p)$ and therefore using that $\tau < \tau_{\rm hea}$ and  Lemma~\ref{lem:sparsity} we obtain $|[0,\tau ]\cap I_j \cap \mathcal S |\le p^{-1/3}$. Since $\mathcal S \cap I_j$ is predictable, we can argue as in the proof of Proposition~\ref{prop:coup-for-variance} that the probability to discover an element of $\mathcal T$ in the first $\lfloor p^{-1/3} \rfloor $ elements of $\mathcal S\cap I_j$ is bounded by $p^{2/3}$. This gives
\begin{equation*}
\mathbb P (\mathcal B _j \mid \mathcal F _{t+t_j-1}) = \mathbb P \big( [0,\tau ]\cap I_j \cap \mathcal S \cap \mathcal T \neq \emptyset \mid \mathcal F _{t+t_j-1} \big) \le p^{2/3}.
\end{equation*}
Thus, using that $\mathcal A _j$ is measurable in $\mathcal F _{t+t_j-1}$ we obtain that $\mathbb P (\mathcal A _j \cap \mathcal B_j \mid \mathcal F _t)\le p^{1/3}j^{-1}$. Hence on the event that $t$ is relaxed we have
\begin{equation*}
\begin{split}
    \mathbb P \big(  \tau _{\rm hit}(t)\le t+T \mid \mathcal F _t \big)  \le \sum_{j=1}^{j_1} p^{1/3}j^{-1} \le 1/4,
\end{split}
\end{equation*}
where in the last inequality we used that $T\le e^{\log ^2(1/p)}$.
\end{proof}

\subsection{Heavy blocks are unlikely}
In this section we prove Proposition~\ref{prp:noheavy} which states that heavy blocks are unlikely to appear quickly.

The main technical ingredient we use is the following estimate which shows that the walk is unlikely to visit any block after a sufficiently long time.
\begin{lemma}\label{claim:box}
If $\Hyp(T)$ holds for some $T\le \bar T$, then for any $r\ge p^{-1}$ and any block $B$ of side length $r$ we have that
    \begin{equation*}
        \mathbb P \big( \exists  t\in [p^{1-3\epsilon }r^{2},T], \  X(t)\in B \big) \le 1/4.
    \end{equation*}
\end{lemma}

\begin{proof}
    Fix $r\ge p^{-1}$ and a block $B$ of side length $r$. Let $B^+:= \{x\in \mathbb Z ^d : d(x,Q) \le 2r \log ^8(1/p)\}$. Let $j_0:=\lfloor p^{-3\epsilon } \rfloor, j_1:=\lfloor T/\lfloor pr^2\rfloor  \rfloor $ and for $j\in [j_0,j_1+1]$ let $t_j :=j\lfloor pr^2 \rfloor $. We have that
    \begin{eqnarray*}
       \lefteqn{ \mathbb P \big( \exists  t\in [p^{1-3\epsilon }r^{2},T], \  X(t)\in B \big)}
       \\
       &\le& \sum _{j=j_0}^{j_1} \mathbb P \big( X(t_j)\in B^+  \big) +\mathbb P \Big( \max _{s\in [t_j,t_{j+1})} \|X(s)-X(t_j)\| \ge 2r \log ^8 p  \Big)\\
        &\le& \sum _{j=j_0}^{j_1} Cp^{-2\epsilon } j^{-d/2}\log^{8d} (1/p)  +(t_{j+1}-t_j)e^{-2\log ^2(1/p)}\le 1/4,
    \end{eqnarray*}
    where in the second inequality we used Assumption~\hyperref[h1]{\bf{(H1)}} to bound the first probability and Assumption~\hyperref[h2]{\bf{(H2)}} to bound the second probability, and where in the last inequality we used that $d\ge 3$ and $T\le e^{\log^2(1/p)}$.
\end{proof}

The idea is to use Lemma~\ref{claim:box} to show that any block is unlikely to be revisited many times.  Let $B$ be a block of side length $r\ge p^{-1}$. Define a sequence of stopping times $\zeta_0=0$ and
\begin{equation*}
    \zeta_j(B) := \min \big\{ t>\zeta _{j-1}(B) +p^{1-4\epsilon }r^{2} : X(t)\in B \big\}.
\end{equation*}
Next, let $j_*:=\lfloor \log ^3(1/p) \rfloor $ and let $\tau _{\rm visit}(B):=\zeta _{j_*}$ be the first time $W$ visits the block $B$ more that $j_*$ times. Finally, let $\tau _{\rm visit}:=\min _{B}\tau _{\rm visit}(B)$, where the minimum is taken over all blocks of side length $r \ge p^{-1}$.

Almost by definition, $\tau_{\rm visit}$ is a lower bound for $\tau_{\rm hea}$.
\begin{claim}\label{claim:heavy}
We have that $\tau_{\rm hea} \geq \tau_{\rm visit} \wedge \tau_{\rm many}$.
\end{claim}

\begin{proof}
New mirrors in $B$ can be discovered only in the time intervals $[\zeta _j,\zeta _j+p^{1-4\epsilon }r^2]$ for $j\ge 0$. Moreover, by the definition of $\tau _{\rm many}$, for any $j\ge 0$ we have
\[
\big| \mcal{T} \cap [\zeta _j,(\zeta _j+p^{1-4\epsilon }r^2)\wedge \tau _{\rm many}] \big|
\leq p^{2-4\epsilon }\log^3(1/p) r^2.
\]
Thus, for all $t\leq \tau_{\rm visit} \wedge \tau_{\rm many}$ we have
\[
|B \cap \mathcal M(t)| \le j_* p^{2-4\epsilon }\log^3(1/p) r^2 \le  p^{1.9}r^2.
\]
This finishes the proof of the lemma.
\end{proof}

Now we are able to complete the proof of Proposition~\ref{prp:noheavy}.

\begin{proof}[Proof of Proposition~\ref{prp:noheavy}]
By Claim~\ref{claim:heavy}, the proposition will follow by proving
\begin{equation}\label{eq:tau visit is large}
    \mathbb P \big( \tau _{\rm visit} \le T \wedge \tau \big) \le e^{-c\log ^3 (1/p)}.
\end{equation}

To this end, it suffices to prove that for any block $B$ of side length $p^{-1}\le r \le T$ at distance at most $T$ from the origin we have
\begin{equation}\label{eq:4}
    \mathbb P \big( \tau _{\rm visit}(B) \le T \wedge \tau \big) \le e^{-c\log ^3 (1/p)}.
\end{equation}
  Indeed, a block $B$ at distance larger than $T$ from the origin satisfies $\tau _{\rm visit}(B) > \zeta _1(B) \ge T$ as the walk cannot reach it by time $T$. Moreover, a block $B$ of side length $r\ge T$ satisfies $\tau _{\rm visit}(B)> \zeta _1(B)\ge p^{1-4\epsilon }r^2\ge T$ since $T\ge p^{-1}$. Thus, the bound in \eqref{eq:tau visit is large} follows from \eqref{eq:4} and a union bound as there are at most $(4T)^{d+1}$ blocks $B$ of side length $p^{-1}\le r \le T$ at distance at most $T$ from the origin.

From now on, we fix such a block $B$ and establish the estimate \eqref{eq:4}. It suffices to prove that for all $j\ge 1$ we have
\begin{equation}
\label{eq:zetaj-ineq}
    \mathbb P \big( \zeta _j \le (\zeta _{j-1}+T)\wedge \tau
 \mid \mathcal F _{\zeta _{j-1}}\big) \le 1/2
\end{equation}
(where the $\zeta_j$ are associated with the fixed block $B$).
Indeed, if this bound holds then
\begin{equation*}
    \mathbb P \big( \tau _{\rm visit} (B) \le T \wedge \tau \big) \le \mathbb P \big( \forall j\le j_*,\ \zeta _j \le (\zeta _{j-1}+T)\wedge \tau \big) \le (1/2)^{j_*} \le e^{-c\log ^3(1/p)}.
\end{equation*}

Next, let $\eta_ j := \min \{t\ge \zeta _{j-1} : t\in \mathcal R\}$ be the first relaxed time after $\zeta _{j-1}$.
To establish~\eqref{eq:zetaj-ineq} it suffices to show that
\begin{equation}\label{eq:3}
    \mathbb P \big( \zeta _j \le (\zeta _{j-1}+T)\wedge \tau   \mid \mathcal F _{\eta _j }\big) \le 1/2,
\end{equation}
since \eqref{eq:zetaj-ineq} follows by taking the conditional expectation of both sides of \eqref{eq:3} with respect to $\mathcal F _{\zeta _{j-1}}\subset \mathcal F _{\eta _{j}}$. The inequality in \eqref{eq:3} clearly holds (even with $1/2$ replaced by $0$) on the event $\{\eta_j>\zeta_{j-1}+t_*\}\in \mathcal F _{\eta _j}$ since on this event by the definition of $\tau _{\rm rel}$ we have $\tau \leq \tau_{\rm rel} \leq \zeta_{j-1}+t_* < \zeta_j$. Hence, we may work on the event that $\eta_j\leq \zeta_{j-1}+t_*$.

Finally, to prove \eqref{eq:3} on the event $\{\eta_j\leq \zeta_{j-1}+t_*\}$ we couple the walk $W$ after time $\eta _j$ with an independent mirror walk. More precisely, let $W'$ be a regenerated mirror walk starting from $X'(0)=X(\eta _j)$ in direction $V'(0)=V(\eta _j)$ that is driven by $\{\tilde{W}(\eta _j +s)\}_{s\ge 0}$. Recall from Claim~\ref{claim:cou} that $X(t+\eta _j)=X'(t)$ for all $t\le \tau _{\rm hit}(\eta _j)-\eta _j$. Thus, on the event $\eta _j \le \zeta _{j-1}+t_*$ we obtain
\begin{equation*}
\begin{split}
    \mathbb P \big( \zeta _j \le (\zeta _{j-1}+T)\wedge   \tau
 \mid \mathcal F _{\eta _j }\big) &\le \mathbb P \big( \exists t\in [\zeta _{j-1}+p^{1-4\epsilon }r^2, (\zeta _{j-1}+T)\wedge \tau ], \ X(t) \in B \mid \mathcal F _{\eta _j }\big) \\
 &\le \mathbb P \big( \exists t\in [\eta _j+p^{1-3\epsilon }r^2, (\eta _{j}+T)\wedge \tau ], \ X(t) \in B \mid \mathcal F _{\eta _j }\big)\\
 \le \mathbb P\big( \exists t\in & [p^{1-3\epsilon } r^{2},T], \  X'(t)\in B \mid \mathcal F _{\eta _j} \big) +\mathbb P \big(\tau _{\rm hit}(\eta _j) \le (\eta _j +T)\wedge \tau  \mid \mathcal F _{\eta _j} \big).
\end{split}
\end{equation*}
Note that $X'$ is independent of $\mathcal F _{\rm \eta _j}$ except for its starting point and starting direction, and therefore we may use Lemma~\ref{claim:box} to bound the first probability by $1/4$. The second probability is bounded by $1/4$ by Proposition~\ref{prp:dont-hit}. This completes the proof of  \eqref{eq:3}.
\end{proof}

\subsection{There are many relaxed times}
The fundamental property of $\Ss$ that we use is that, as long as you avoid discovering mirrors in $\Ss$, you remain relaxed.
\begin{lemma}
\label{lem:ST-to-R}
Suppose that $t_1\in\mcal{T}$ is a time at which the walk changes direction, meaning that  $V(t_1-1)\not=V(t_1)$.
If $\mcal{T} \cap [t_1,t_2]\cap \Ss =\emptyset$ and $t_1<t_2<\tau_{\rm few}\wedge\tau_{\rm hea}$, then $[t_1,t_2] \subset \mcal{R}$.
\end{lemma}
\begin{proof}
Suppose by contradiction that $t_1\le s \le t_2$ is the first time after $t_1$
that is not relaxed, that is
\[
s = \min [t_1,t_2] \setminus \mcal{R}.
\]
Since $s \leq \tau_{\rm hea}$, it must be that $s$ is not locally relaxed and therefore there exists $0<r\leq t_*$ such that
\begin{equation}\label{eq:baluga}
X(s)+rV(s)\in \mcal{M}(s).
\end{equation}
Let $t_1\leq s^-\leq s$ be the last time that the walk changed direction before $s$, which exists because $t$ is such a time.

We claim that $s^-\in \mathcal T$. Indeed, it holds by definition of $t_1$ if $s^-=t_1$.  On the other hand, if $s^->t_1$ and $s^-\notin \mathcal T$ then $X(s^-)$ is a previously discovered mirror but then $s^--1\notin \mathcal{R}$ contradicting the minimality of $s$. Thus, $s^-\in \mathcal T$.  But $s\leq \tau_{\rm few}$, so $|s-s^-|\leq t_*$.  By \eqref{eq:baluga} and the definition of $s^-$ we have $X(s^-)+V(s^-) (s-s^-+r) \in \mathcal M(s^-)$ and $s-s^-+r\le t_*+r\le 2t_*$ and so $s^-\in S$ contradicting the assumption that $\mcal{T} \cap [t_1,t_2]\cap \mathcal S = \emptyset$.  This completes the contradiction, and the proof is done.
\end{proof}

Finally, we prove Proposition~\ref{prp:relaxed}.

\begin{proof}[Proof of Proposition~\ref{prp:relaxed}]
Recall the definition of $\tau _{\rm rel}$ from \eqref{eq:tau-rel} and fix $t\ge t_*$. We would like to bound the probability that $\tau _{\rm rel}=t\le \tau $. Define the following sequence of stopping times. Let $\xi _0:=t-t_*$ and for all $j\ge 1$
\begin{equation*}
    \xi _{j+1} : =\min \big\{ s>\xi _j +2p^{-1}: s\notin \mathcal R \big\}.
\end{equation*}
On the event $\tau \ge t$, the set $[t-t_*,t]\setminus \mathcal R$ is contained in the union of intervals $[\xi _j,\xi _j +2p^{-1}]$ and therefore on the event that $\tau \ge t$ and $|\{j\ge 0: \xi _j < \tau \}| \le 0.05\log ^3(1/p)$ we have that $|[t-t_*,t]\setminus \mathcal R| \le 0.1t_*$ and therefore $\tau _{\rm rel}\neq t$. Thus, the statement of the proposition will follow from the estimate
\begin{equation*}
    \mathbb P \big( \tau \ge t \text{ and }  |\{j: \xi _j < t\wedge \tau  \}| \ge 0.05\log ^3(1/p) \big) \le e^{-c\log ^3(1/p)}.
\end{equation*}
and a union bound over $t\le T$. The last estimate easily follows from the conditional bound
\begin{equation}\label{eq:7}
    \mathbb P \big( \xi _{j+1}  < t\wedge \tau  \mid  \mathcal F _{\xi _j } \big) \le 4/5.
\end{equation}

Let $\eta _j:=\min \{s>\xi _j \mid s\in \mathcal T\}$ and recall that  $\tilde{W}$ is the driving walk of $W$. Define the events
\begin{equation*}
    \mathcal A _1:=\{\eta _j \le \xi _j+2p^{-1} \},  \quad \mathcal A _2 := \{\tilde{V}(\eta _j)\neq \pm V(\eta _j-1)\},  \quad \mathcal A _3: = \big\{ \mathcal T \cap S \cap (\xi _j, t\wedge \tau ] =\emptyset \big\},
\end{equation*}
and $\mathcal A :=\mathcal A _1 \cap \mathcal A _2 \cap \mathcal A _3$.

Next we show that on the event $\mcal{A}$, we have that $\xi_{j+1}> t\wedge \tau$.  We may assume that $\xi_j + 2p^{-1} < t\wedge \tau$ as otherwise the claim is trivial, so on the event $\mcal{A}_1$ this implies that $\eta_j\in(\xi_j,t\wedge \tau]$.
First we claim that $V(\eta_j-1)\not=V(\eta_j)$  on this event.  To see this, note that $\eta_j\in \mcal{T} \cap (\xi_j,t\wedge \tau]$ so that on the event $\mcal{A}_3$ we have that $\eta_j\not\in S$.
 Therefore, since $\eta_j<\tau \leq \tau_{\rm few}$, it follows that $X(\eta_j)\not\in \{X(s): s\in [\alpha(\eta_j),\eta_j)\}$ (as in~\eqref{eq:history-in-union}).
 Then, by the construction of the coupling and the fact that
 $\tilde{V}(\eta_j) \not= -V(\eta_j-1)$, it follows that $V(\eta_j)=\tilde{V}(\eta_j)\not= V(\eta_j-1)$, proving that $V(\eta_j)\not= V(\eta_j-1)$.
 Now we can apply Lemma~\ref{lem:ST-to-R} with $t_1=\eta_j < \xi_j+2p^{-1}$ and $t_2=t\wedge \tau$ to conclude that $[\xi_j+2p^{-1},t\wedge\tau]\subset\mcal{R}$ on the event $\mcal{A}$, and therefore $\xi_{j+1} > t\wedge \tau$.

What remains is to bound the probability of the event $\mcal{A}$ using a union bound.

The bound on $\mcal{A}_1$ is completely elementary:
\begin{equation}\label{eq:10}
\mathbb P (\mathcal A _1  \mid \mathcal F _{\xi _j} ) \ge 1- (1-p)^{2p^{-1}} \ge 1-e^{-2} \geq \frac67.
\end{equation}
Likewise, we have for $d\geq 4$
\[
\Prob(\mcal{A}_2 \mid  \mcal{F}_{\xi _j}) \geq \frac{2d-3}{2d-1}
\geq \frac57
\]

Finally to bound $\mcal{A}_3$ we use Lemma~\ref{lem:sparsity} to see that
\begin{align*}
|\Ss \cap (\xi _j, t\wedge \tau ]|
&\leq 2(|\mcal{M}(X(\xi _j)) \cap B_{3t_*}(X(\xi _j))| +
|\mcal{T}\cap [\xi _j,t\wedge\tau])^2 \\
&\le 2p^{-0.2}.
\end{align*}
The bound on $|B_{3t_*}(X(\xi _j))|$
comes from the fact that $\xi _j\leq \tau_{\rm hea}$,
and the bound on $\mcal{T}\cap[\xi _j,t\wedge\tau]$ comes from the fact that $\tau < \tau_{\rm many}$.

Exposing the walk from $\xi _j$ to $\tau$, each element of $S\cap (\xi _j, t\wedge\tau ]$ is in $\mathcal T $ with probability $p$ and therefore
\begin{equation}\label{eq:9}
    \mathbb P \big( \mathcal A _3 \mid \mathcal F _{\xi _j} \big) \ge 1-\lceil pt_* \rceil \sqrt{p} \ge 0.9.
\end{equation}
\end{proof}

\section{Concatenation of walks}\label{sec:conc}

Let $W_1^*$ and $W_2^*$ be two independent regenerated mirror walks of lengths $t_1,t_2\ge 0$, respectively, both starting from the origin from a uniform random direction. We define the concatenation $W_c=(X_c,V_c)$ of $W_1^*$ and $W_2^*$ by
\begin{equation}\label{eq:defofcon}
     X_c(s):=\begin{cases}
    \quad \quad  X_1^*(s) \quad \quad \quad \quad   \ \  &s\in [0,t_1]
\\
 X_1^*(t_1)+X_2^*(s-t_1) \quad  &s \in [t_1, t_1+t_2]
    \end{cases}.
\end{equation}
Our goal is to show that the concatenation $W_c$ can be coupled with the regenerated mirror walk $W$ of length $t_1+t_2$ starting from direction $e_1$.

Recall from Section~\ref{sec:driving} and Section~\ref{sec:regenerated} how a mirror walk and a regenerated walk are driven by a non-backtracking walk $\tilde{W}$. In the natural coupling of the concatenation $W_c$ and the regenerated walk $W$, we simply drive both of these processes using the same non-backtracking walk $\{\tilde{W}(s)\}_{s\ge 0}$. Namely, we drive the regenerated walk $\{W_1^*(s)\}_{s\in [0,t_1]}$ using $\{\tilde{W}(s)\}_{s\in [0,t_1]}$, the regenerated walk $\{W_2^*(s)\}_{s\in [0,t_2]}$ using $\{\tilde{W}(t_1+s)\}_{s\in [0,t_2]}$ and the regenerated walk $\{W(s)\}_{s\in [0,t_1+t_2]}$ using $\{\tilde{W}(s)\}_{s\in [0,t_1+t_2]}$. The next result shows that $X$ and $X_c$ will remain close with high probability.

\begin{proposition}
\label{prop:conc}
    Let $t_1,t_2\ge 0$ with $t_1+t_2\le T+1$ and suppose $\Hyp(T)$ holds for some $T\le \bar T$. Let $W_1^*,W_2^*$ be two independent regenerated mirror walks each with uniform random initial velocity. Let $W_c$ be the concatenation of $W_1^*$ and $W_2^*$ and let $W$ be a regenerated mirror walk of length $t_1+t_2$ with initial velocity $e_1$. Then, there is a coupling of $W_c$ and $W$ such that
    \begin{equation}\label{e.dor+2}
        \mathbb P \Big(  \max _{t\le t_1+t_2} \|X(t)-X_c(t)\| \ge p^{-1} \log^7 (1/p) \Big) \le e^{-c\log ^3(1/p)}.
    \end{equation}
The same result holds if $W$ is replaced by a regenerated mirror walk $W^*$ of length $t_1+t_2$ with uniform random initial velocity.
\end{proposition}

\begin{proof}
First, observe that by extending $W$ and $W_2$ if needed, we may assume that $t_1+t_2=T+1$.

In order to couple $W$ and $W_c$ we simply drive $W$ and $W_c$ using the same driving mirror walk $\tilde{W}$. Namely, the walk $W_1^*$ is driven by $\{\tilde{W}(s)\}_{s\le t_1}$, the walk $W_2^*$ is driven by $\{\tilde{W}(t_1+s)\}_{s\le t_2}$ and the walk $W$ is driven by $\{\tilde{W}(s)\}_{s\le t_1+t_2}$.  The distance between the two walks might grow when their velocities are different. But if $\tau_{\rm rel}$ (and its analogous stopping time for $W_c$) has not yet occurred, we can quickly find a time that is relaxed with respect to both $W$ and $W_c$.
Quickly after this time, there is a positive chance that the two walks will be coupled perfectly.  We therefore need Corollary~\ref{cor:taurel} as an input to ensure the existence of many relaxed times.

Recall that $\mathcal R$ is the set of relaxed times with respect to $W$ and let $\mathcal R_1\subseteq [0,t_1]$ and $\mathcal R _2\subseteq [0,t_2]$ be the sets of relaxed times with respect to  $W_1^*$ and $W_2^*$ respectively. Let $\mathcal R_c :=\mathcal R _1 \cup (t_1+\mathcal R _2)$. Next, we define a sequence of stopping times $\zeta _1,\zeta _2 \dots $ and $\zeta _0', \zeta _1', \zeta _2 '\dots $ in the following way. We let $\zeta _0':=0$ and for $i\ge 1$ we let
\begin{equation*}
    \zeta _i:=T\wedge \min \big\{ t>\zeta _{i-1}' : t\in \mathcal R \cap \mathcal R _c   \big\} \quad \text{and} \quad \zeta _i': = T\wedge \min \big\{ t>\zeta _i+2p^{-1} : V(t)\neq V_c(t) \big\}.
\end{equation*}
For any $s\in (\zeta _i+2p^{-1}, \zeta _i')$ we have that $V(s)=\tilde{V}(s)$.  Therefore any $s\leq T$ for which $V(s)\not= \tilde{V}(s)$ is contained in an interval of the form $[\zeta_{i-1},(\zeta_i+2p^{-1})\wedge T]$, so
\begin{equation}\label{eq:33}
\begin{split}
    \max _{t\le T}\|X(t)-X_c(t)\|&\le \sum _{t<T} \|V(t)-V_c(t)\| \\
    &\le 2\sum _{i=1}^{\infty } \mathds 1 \{\zeta _{i-1}'< T\} \cdot \big( (\zeta _i+2p^{-1})\wedge T-\zeta _{i-1}' + 1\big)
\end{split}
\end{equation}
In the remainder of the proof we bound the number of non-zero terms in the sum and then bound the size of each of them.

\medskip

\noindent \textit{Bound on number of nonzero terms:}\\
Let $N:= |\{ i\ge 0 : \zeta _i'<T \}|$ be the number of nonzero terms in~\eqref{eq:33}. We will show that
\begin{equation}\label{eq:Nbd}
\Prob \big( N> \log^3(1/p) \big) \leq e^{-c\log^3(1/p)}.
\end{equation}
To this end, recall the definition of $\tau $ from \eqref{eq:deftau} and let $\tau _1$ and $\tau _2$ be the analogous stopping times for the walks $W_1^*$ and $W_2^*$ respectively. Let $\tau _c:=\tau _1 \wedge (t_1+\tau _2)$ and observe that by Corollary~\ref{cor:taurel} applied to the walks $W,W_1^*$ and $W_2^*$ we have that
\begin{equation}\label{eq:tau,tau_c}
 \mathbb P (\tau \wedge \tau _c\le T) \le e^{-c\log ^3(1/p)}.
\end{equation}

Next, we claim that for all $i\ge 1$, on the event  $\{\zeta _i \le t_1\}$ we have
\begin{equation}\label{eq:12}
   \mathbb P \big( \zeta_{i}'\geq t_1 \wedge \tau \wedge \tau _c \mid \mathcal F _{\zeta _i} \big)\ge 0.1
\end{equation}
and on the event $\{t_1\le \zeta _i \le  T\}$ we have
\begin{equation}\label{eq:13}
    \mathbb P \big( \zeta_{i}'\geq T\wedge \tau \wedge \tau _c \mid \mathcal F _{\zeta _i} \big)\ge 0.1.
\end{equation}
We prove \eqref{eq:13} while the proof of \eqref{eq:12} is identical and is omitted. Recall the definition of $\tau _{\text{hit}}(t)$ in \eqref{eq:def_tau_hit} for the walk $W$, and let $\tau _{\rm hit}^{2}(t)$ be the analogous stopping time for the walk $W_2^*$. For $t\ge t_1$, we let $\tau _{\rm hit}^c(t):=t_1+\tau _{\rm hit}^2(t-t_1)$ be the first time after $t$ in which $W_c$ interacts with its relevant past before $t$. Let $\zeta _i^+:=\min \{t>\zeta _i :t \in \mathcal T \}$ and observe that by the second part of Proposition~\ref{prp:dont-hit} applied to both $W$ and $W_2^*$ we have that
\begin{equation}\label{eq:71}
    \mathbb P \big( \tau _{\rm hit}(\zeta _i^+) \wedge \tau _{\rm hit} ^c (\zeta _i^+) \ge T\wedge \tau \wedge \tau _c \mid \mathcal F _{\zeta _i}  \big) \ge 1/2.
\end{equation}
Moreover, by part (2) of the definition of a relaxed time, if $\zeta _i^+\le \zeta _i+2p^{-1}$, then both $W$ and $W_c$ will encounter a new mirror at time $\zeta _i^+$ and by \eqref{eq:prob of coupling}, each will have a chance of at least $(2d-2)/(2d-1)$ to be coupled with the direction $\tilde{V}(\zeta _i^+)$ of the driving process. We obtain
\begin{equation}\label{eq:81}
\begin{split}
    \mathbb P \big( \zeta _i^+\le \zeta _i +p^{-1} , \ V(\zeta _i^+)=V_c(\zeta _i^+) \mid \mathcal F _{\zeta _i}  \big) &\ge \frac{2d-3}{2d-1}  \cdot \mathbb P \big( \zeta _i^+\le \zeta _i +2p^{-1}  \mid \mathcal F _{\zeta _i}  \big)\\
    &\ge \frac{5}{7} \cdot \big( 1-(1-p)^{\lfloor 2p^{-1} \rfloor } \big) \ge 0.6,
\end{split}
\end{equation}
where in the second inequality we used that $d\ge 4$ and in the last inequality we used that $(1-p)^{2p^{-1}}\approx e^{-2}$ for small $p$. On the intersection of the events in \eqref{eq:71} and \eqref{eq:81} we have that $\zeta _i'\ge T\wedge \tau \wedge \tau _c$ which completes the proof of \eqref{eq:13}. Using \eqref{eq:13} and letting $i_1:= \lfloor \log ^3(1/p) \rfloor $ and $i_2:=2\lfloor \log ^3(1/p) \rfloor$ we obtain
\begin{equation*}
    \mathbb P \big( \forall i \in [i_1,i_2], \  t_1 \le \zeta _i' <T\wedge \tau \wedge \tau _c  \big) \le \prod _{i\in (i_1,i_2]} \mathbb P \big( \zeta _i' <T \wedge \tau \wedge \tau _c \mid t_1 \le \zeta _i <T \big)\le e^{-c\log ^3(1/p)}.
\end{equation*}
Similarly, using \eqref{eq:12} we obtain
\begin{equation*}
    \mathbb P \big( \forall i \in [1,i_1], \  \zeta _i' <t_1\wedge \tau \wedge \tau _c  \big) \le \prod _{i\in [1,i_1]} \mathbb P \big( \zeta _i' < t_1 \wedge \tau \wedge \tau _c \mid  \zeta _i <t_1 \big)\le e^{-c\log ^3(1/p)}.
\end{equation*}
Finally, we observe that if $N\ge 2\log ^3(1/p)$ then either $\tau \wedge \tau _c \le T$ or one of the two events above holds. This finishes the proof of~\eqref{eq:Nbd} using \eqref{eq:tau,tau_c}.

\medskip

\noindent \textit{Bound on the size of the terms:}\\
We claim that on the event $\tau \wedge \tau _c >T$ each one of the terms on the right hand side of~\eqref{eq:33} is bounded by $2t_*$. Indeed, recall the definition of $\tau _{\rm rel} \le \tau $ in \eqref{eq:tau-rel} for the regenerated walk $W$, and let $\tau _{\rm rel}^1\le \tau _1$ and $\tau _{\rm rel}^2 \le \tau _2$ be the analogous stopping times for the regenerated walks $W_1^*$ and $W_2^*$ respectively. On the event $\tau \wedge \tau _c >T$, for all $t\le T$ we have that $|\mathcal R \cap [t-t_*,t]| \ge 0.9 t_*$ and $|\mathcal R_c \cap [t-t_*,t]| \ge 0.8 t_*$. It follows that on this event, if $\zeta _{i-1} \le T-t_*$ then $\zeta _i -\zeta _{i-1} \le t^*$, and therefore each one of the terms on the right hand side of~\eqref{eq:33} is bounded by $2t_*$.

Substituting the bounds in \eqref{eq:33}, we obtain that on the event $\{N\le 2 \log ^3(1/p)\} \cap \{\tau \wedge \tau _c >T\}$ we have $\max _{t\le T}\|X(t)-X_c(t)\| \le  2Nt_* \le  Cp^{-1} \log ^6(1/p)$ and therefore by \eqref{eq:Nbd} and \eqref{eq:tau,tau_c} we have
\begin{equation*}
        \mathbb P \Big(  \max _{t\le T} \|X(t)-X_c(t)\| \ge \tfrac13 p^{-1} \log^7(1/p) \Big) \le e^{-c\log ^3(1/p)},
\end{equation*}
which yields \eqref{e.dor+2} since $\|X(T+1)-X(T)\|+\|X_c(T+1)-X_c(T)\|\le 2$.
\end{proof}

Proposition~\ref{prop:conc} directly yields the following by induction.
\begin{corollary}\label{cor:conca}
    Let $n\in \mathbb N$ and let $t_1,\dots t_n\ge 0$ with $t_1+\cdots +t_n\le T+1$ and $\Hyp(T)$ holds. Let $W_1^*,\dots ,W_n^*$ be $n$ independent regenerated mirror walks with uniform random initial velocity. Let $W_c$ be the concatenation of $W_1^*,\dots W_n^*$ and let $W$ be a regenerated mirror walk of length $t_1+\cdots +t_n$ with initial velocity $e_1$. Then, there is a coupling of $W_c$ and $W$ such that
    \begin{equation*}
        \mathbb P \Big(  \max _{t\le t_1+\dots+t_n} \|X(t)-X_c(t)\| \ge np^{-1} \log^7 (1/p) \Big) \le e^{-c\log ^3(1/p)}.
    \end{equation*}
\end{corollary}

\begin{remark}\label{rem:conc}
Proposition~\ref{prop:conc} directly implies that
\begin{equation*}
        \mathbb P \Big(  \max _{t\le T+1} \|X(t)-X^*(t)\| \ge p^{-1} \log^7 (1/p) \Big) \le e^{-c\log ^3(1/p)}.
\end{equation*}
\end{remark}

\section{Completion of the proof: induction step}
\label{sec:induct}
Recall that by Remark~\ref{rem:trivial}, Assumption~$\Hypone{T}$ trivially holds when $T\le p^{-1-2\epsilon/d}$ while Assumption~$\Hyptwo{T}$ trivially holds for $T\le p^{-1} \log^{16} (1/p)$.

\medskip

Assume now that $\Hyp(T)$ holds, which implies the validity of the concatenation results of Section~\ref{sec:conc} up to time $T+1$, which we assume to be smaller than $e^{\log^2 (1/p)}$.
We start by showing a bound on the variance of $W^*$ (as well as \eqref{e.th:main-var}) based on Section~\ref{sec:KP} for smaller times and on Assumption~\hyperref[h2]{\bf{(H2)}} and  concatenation for larger times. We then argue using this variance bound, concatenation, and an anti-concentration argument
that Assumption~$\Hyptwo{T+1}$ holds. Last, by using~\hyperref[h2]{\bf{(H2)}}, a concentration argument, concatenation, and Assumption~$\Hypone{T}$, we prove Assumption~$\Hypone{T+1}$.

\subsection{Effective diffusion of the walk and proof of \eqref{e.th:main-var}}

The aim of this section is to provide a control on the variance of the regenerated walk $W^*$ with random uniform initial velocity, and prove \eqref{e.th:main-var}. Recall that by symmetry, the covariance is a scalar matrix, whose diagonal entries we denote by $\sigma_t^2$.
\begin{lemma}\label{lem:var} If $\Hyp(T)$ holds, we have:
\begin{itemize}
\item For all $p^{-1}\le t\le T$,
\begin{equation}\label{ea:varKP}
ct/p\le \sigma _t^2 \le Ct/p.
\end{equation}
\item In addition, if $ p^{-5/4} \le T$, we have for all $p^{-5/4} \le t\le T$
\begin{equation}\label{ea:varIt}
    \Big| \frac p t \sigma _t^2- \frac{2d-1}{d(d-1)} \Big| \le  Cp^{1/8}\log ^{15}(1/p).
\end{equation}
\end{itemize}
\end{lemma}
\begin{proof}
 For $p^{-1}\le t\ll p^{-4/3}$, \eqref{ea:varKP} is a short-time estimate that holds by the Kesten-Papanicolaou argument in form of Lemma~\ref{lem:driving} and Corollary~\ref{cor:coup-for-variance}.
In the range $p^{-5/4}\le t\le T$, \eqref{ea:varKP} holds as a consequence of \eqref{ea:varIt} which we presently prove. Assume that $T\ge p^{-5/4}$.

\medskip

We start with a preliminary estimate.
Let $p^{-5/4}/3 \le t\le T/2$ and let $W_1^*$ and $W_2^*$ be two independent regenerated walks of length $t$ with a uniform random initial velocity and let $W_c$ be the concatenation of $W_1^*$ and $W_2^*$. By Proposition~\ref{prop:conc}, there is a coupling such that
\begin{equation*}
        \mathbb P \Big(  \|X^*(2t)-X_c(2t)\| \ge p^{-1} \log^7 (1/p) \Big) \le e^{-c\log ^3(1/p)}.
\end{equation*}
Thus, using that $\sigma _{2t}^2=\mathbb E [X^*(2t)_1^2]$ and $2\sigma _{t}^2=\mathbb E [X_c(2t)_1^2]$ (recall the walks are centered and symmetric by reflection) we obtain by Cauchy-Schwarz' inequality  and the triangle inequality in $L^2(\Omega )$
\begin{equation}\label{eq:t2t}
\begin{split}
    |\sigma _{2t}^2-2\sigma _t^2 | =|\E \big[ X^*(2t)_1^2-X_c(2t)_1^2\big] |= |\E \big[ \big( X^*(2t)_1-X_c(2t)_1 \big) \big( X^*(2t)_1+X_c(2t)_1 \big) \big] |\\
    \le  \sqrt{\E \big[ \big( X^*(2t)_1-X_c(2t)_1 \big)^2 \big]} (\sigma _{2t}+\sqrt 2 \sigma _t)
    \le \sqrt{t} \cdot p^{-3/2} \log ^{15}(1/p),
\end{split}
\end{equation}
where in the last inequality we used that
\begin{equation}\label{e.bd-var0}
\mathbb E[\|X^*(t')\|^2]\le C (t'/p) \log ^{16}(1/p)
\end{equation}
for $t'=t,2t$ to bound $\sigma_{2t}+\sigma_{t}$.
Indeed, by symmetry and a union bound on all possible initial velocities for $X^*$,
Assumption~$\Hyptwo{T}$ on $X(t')$ yields
\begin{equation}\label{e.H2*}
\mathbb P \Big(\|X^*(t')\|\ge \sqrt{t'/p}\log^8(1/p)\Big)\le C e^{-\log^2 (1/p)},
\end{equation}
from which \eqref{e.bd-var0} follows.

Next, we prove the desired claim by iteration.
Let $p^{-5/4} \le t\le T$ and define inductively $t_0=t$ and for $j\ge 1$
\begin{equation*}
    s_j:= 2\lfloor t_{j-1}/2 \rfloor  \quad \text{and} \quad t_j:= s_j/2.
\end{equation*}
In words, we define $s_j$ to be the even integer in $\{t_{j-1},t_{j-1}-1\}$ and then divide it by two in the definition $t_j$. Let $m$ be the first integer $j$ for which $t_j \le p^{-5/4}$.
By the triangle inequality,
\begin{equation}\label{a001}
\begin{split}
    \Big| \frac{\sigma _t ^2}{t}-\frac{\sigma _{t_m}^2}{t_m} \Big| &\le \sum _{j=1}^m \Big| \frac{\sigma _{t_{j-1}}^2}{t_{j-1}}- \frac{\sigma _{s_j}^2}{s_j}\Big|+\Big| \frac{\sigma _{s_j}^2}{s_j}- \frac{\sigma _{t_j}^2}{t_j}\Big|.
\end{split}
\end{equation}
For the first right-hand side term, since $|t_{j-1}-s_j|\le 1$ we have $|X^*(t_{j-1})_1-X^*(s_j)_1|\le 1$ so that by \eqref{e.bd-var0}
each summand satisfies
\[
\Big| \frac{\sigma _{t_{j-1}}^2}{t_{j-1}}- \frac{\sigma _{s_j}^2}{s_j}\Big|
\,\le\,\Big| \frac{\sigma _{t_{j-1}}^2 - \sigma _{s_j}^2}{t_{j-1}}\Big|+ \frac{\sigma _{s_j}^2}{t_{j-1} s_j}\,\le  \frac{\sigma _{t_{j-1}}+ \sigma _{s_j}}{t_{j-1}}+ \frac{\sigma _{s_j}^2}{t_{j-1} s_j}
\,\le\,  \frac {C\log ^8(1/p)}{ p^{1/2}\sqrt{s_j}} .
\]
For the second right-hand side term of \eqref{a001} we appeal to \eqref{eq:t2t}. Summing over $j$, we then obtain
\begin{equation}\label{a002}
\begin{split}
 p \Big| \frac{\sigma _t ^2}{t}-\frac{\sigma _{t_m}^2}{t_m} \Big| &\le Cp^{1/8} \log ^{15}(1/p).
\end{split}
\end{equation}
To conclude the proof of the statement, we use the triangle inequality in form of
\begin{equation*}
\begin{split}
    \Big| \frac{p}t\sigma _t ^2-\frac{2d-1}{d(d-1)} \Big| & \le \, p\Big| \frac{\sigma _t ^2}{t}-\frac{\sigma _{t_m}^2}{t_m} \Big|+p\Big| \frac{ \sigma _{t_m}^2}{t_m}-\frac{\tilde\sigma _{t_m}^2}{t_m} \Big|+\Big|\frac{p }{t_m}\tilde \sigma _{t_m}^2-\frac{2d-1}{d(d-1)} \Big|,
\end{split}
\end{equation*}
where $\tilde \sigma_t^2$ is the variance of the driving walk, and we appeal to \eqref{a002}, and to Lemma~\ref{lem:driving} and Corollary~\ref{cor:coup-for-variance} at time
$\frac13  p^{-5/4}\le t_m \le p^{-5/4} $ to control the right-hand side.
\end{proof}
The desired bound \eqref{e.th:main-var} follows from Lemma~\ref{lem:var}, the identity $\mathbb E[\|X^*(t)\|_2^2]=d\cdot \sigma_t^2$, and
\begin{equation*}
|\mathbb E[\|X^*(t)\|_2^2]-\mathbb E[\|X(t)\|_2^2]|\,\le\, \sqrt{t} \cdot p^{-3/2} \log ^{15}(1/p),
\end{equation*}
which, arguing as for \eqref{eq:t2t}, is a consequence of Remark~\ref{rem:conc}, Assumption~$\Hyptwo{T}$ on $X(t)$, and~\eqref{e.H2*}.

\subsection{Traveling far is unlikely: proof of Assumption~$\Hyptwo{T+1}$}

In this section we prove that Assumption~$\Hyptwo{T+1}$ holds.
The proof makes use of the following standard concentration result for a sum of independent variables (see, e.g., \cite[Theorem~3.4]{chung2006concentration}).
\begin{lemma}\label{lem:freedman}
    Let $n\ge 1$ be an integer and let $M>0$. Let $Z_1,\dots ,Z_n$ be independent random variables such that for all $i\le n$ we have $\mathbb E [Z_i]=0$ and $|Z_i|\le M$ almost surely. Letting $S=\sum _{i=1}^n Z_i$ we have for all $x>0$
    \begin{equation*}
        \mathbb P \big( |S|\ge x \big) \le 2\exp \Big(- \frac{x^2}{2\var (S)+Mx} \Big).
    \end{equation*}
\end{lemma}
Now we prove~$\Hyptwo{T+1}$.
Fix $s\le t\le T+1$ with $t-s\ge p^{-1}$. Let $t':=\lfloor \sqrt{(t-s)/p} \rfloor $ and let $n:=\lfloor\sqrt{p(t-s)}\rfloor$. Let $W_1^*,\dots ,W_n^*$ be independent mirror walks starting from a uniform random direction and of lengths $s$ for $W_1^*$, $t'$ for $\{W_i^*\}_{i=2...n-1}$, and $t-s-(n-2)t'\in [t',2t']$ for $W_n^*$. Let $W_c$ be the concatenation of these walks (a walk of length $t$). Since $\Hyp(T)$ holds, by Corollary~\ref{cor:conca}, there is a coupling of $W$ and $W_c$ such that
\begin{equation}\label{a0}
\mathbb P \Big( \max _{s'\le t} \| X(s')-X_c(s')\| \ge n p^{-1} \log ^7(1/p) \Big) \le \exp (-c\log ^3(1/p)).
\end{equation}
Next, note that by Lemma~\ref{lem:var}
 we have that $\var(X_c(t)_i-X_c(s)_i) \le C(t-s)/p$ for any coordinate $i\le d$. Thus, by Lemma~\ref{lem:freedman} and a union bound over the coordinates we have
\begin{eqnarray*}
    \mathbb P \big( \|X_c(t)-X_c(s)\| \ge \sqrt{(t-s)/p} \log ^3(1/p) \big) &\le &2d\exp \Big(  \frac{-c(t-s)p^{-1} \log ^6(1/p)}{ C(t-s)/p +t' \sqrt{(t-s)/p} \log ^3(1/p) }\Big) \\
    &\le& \exp (-c\log ^3(1/p)).
\end{eqnarray*}
In combination with \eqref{a0} it entails
\begin{equation*}
    \mathbb P \big( \|X(t)-X(s)\| \ge 3\sqrt{(t-s)/p} \log ^7(1/p) \big)\le  \exp (-c\log ^3(1/p)).
\end{equation*}
The last estimate clearly holds also when $t-s\le p^{-1}$ and therefore    $\Hyptwo{T+1}$ holds.

\subsection{Transition probabilities: proof of Assumption~$\Hypone{T+1}$}

In this section we prove that $\Hypone{T+1}$ holds.
To this aim, we need the following anti-concentration result, the proof of which we postpone to the appendix.
\begin{lemma}\label{lem:felipe}
Let $n\ge 1$ be an integer and let $M>1$. Let $Z_1,\dots ,Z_n$ be i.i.d.\ random variables on $\Real^d$ satisfying for all $j\le n$
\begin{equation*}
\Expec Z_j = 0 \quad
\cov (Z_j) = \operatorname*{Id}  \quad \Expec \|Z_j\|^3 \leq M.
\end{equation*}
Let $S_n = \sum_{j=1}^n Z_j$. For any $z\in \mathbb R ^d$ and any $r\ge M$ we have that
\[
\P \big( \|S_n -z\| \le r \big) \leq C  r^d  n^{-d/2}.
\]
\end{lemma}
We will apply this to $W$ to get an anti-concentration result for the walk $X(t)$:
\begin{corollary}\label{cor:global}
    For any $t\in [p^{-1},T+1]$ any $r\ge t^{1/3}p^{-2/3}  \log^{25}(1/p)$ and every $z\in \mathbb R ^d$
    \begin{equation*}
        \mathbb P \big( \|X(t)-z\| \le r \big) \le Cr^d(t/p)^{-d/2}.
    \end{equation*}
\end{corollary}
Before we prove Corollary~\ref{cor:global}, let us show how to use it to complete the proof of~$\Hypone{T+1}$. By Remark~\ref{rem:trivial}, Assumption~$\Hypone{T+1}$ trivially holds if $T+1< p^{-1-2\epsilon/d}$, and we may thus assume that $T+1\ge p^{-1-2\epsilon/d}$.
Let $t_2:=T^{2/3}p^{-1/3}$ and $t_1:=T+1-t_2$, and let  $r_1:=T^{1/3}p^{-2/3} \log^{25}(1/p)$ and $r_2:=2p^{-1} \log ^7(1/p)$. Let $W_1^*,W_2^*$ be two independent regenerated walks of lengths $t_1,t_2$ and let $W_{c}$ be the concatenation of these walks. By $\Hyp(T)$ and Proposition~\ref{prop:conc}, for any $z\in \mathbb Z ^d$ we have that
    \begin{eqnarray*}
      \lefteqn{  \mathbb P \big( \|X(T+1) - z\|\le p^{-1} \big)}
      \\
      & \le& \mathbb P \big(\|X_c(T+1)-z\|\le r_2 \big) +e^{-c\log ^3(1/p)} \\
        &\le &\mathbb P \big( \|X_1^*(t_1)-z\| \le r_1  \big) \mathbb P \big( \|X_2^*(t_2)+X_1^*(t_1)-z\|\le r_2 \mid \|X_1^*(t_1)-z\| \le r_1  \big)
        \\
        &&+\mathbb P \big( \|X_1^*(t_1)-z\| > r_1  \big)\cdot e^{-\log ^2(1/p)} +e^{-c\log ^3(1/p)},
    \end{eqnarray*}
where in the last inequality we used Assumption~$\Hyptwo{t_2}$ to bound the probability that $\|X_2^*(t_2)\|\ge r_1-r_2$.
Using now Assumption~$\Hypone{t_2}$  for the walk $W^*_2$ and Corollary~\ref{cor:global}   for the walk $W_1^*$, this yields
\begin{eqnarray*}
      \lefteqn{\mathbb P \big( \|X(T+1) - z\|\le p^{-1} \big)}
      \\
        &\le& Cr_1^d ((T+1)/p)^{-d/2} \cdot \log ^{7d}(1/p)p^{-\epsilon }(pt_2)^{-d/2} e^{ (\log t_2)^{1/3}}
        \\
        &=& C (p(T+1))^{-d/2} p^{-\epsilon } \log^{25+7d}(1/p)  e^{(\log t_2)^{1/3}}\le (p(T+1))^{-d/2} p^{-\epsilon } e^{(\log (T+1))^{1/3}},
\end{eqnarray*}
and $\Hypone{T+1}$ is proved.

\medskip

We conclude with the proof of Corollary~\ref{cor:global}.
\begin{proof}[Proof of Corollary~\ref{cor:global}]
Let $p^{-1} \le t \le T+1$, let $r\ge t^{1/3}p^{-2/3}  \log^{25}(1/p)$ and let $z\in \mathbb R^d$. Let $n:=\lfloor t^{1/3}p^{1/3}\rfloor $ and let $t':=\lfloor t/n \rfloor $. Let $W_1^*,\dots ,W_n^*,W_{n+1}^*$ be independent regenerated walks, all of length $t'$ except $W_{n+1}^*$ which is of length $t-nt'\le t'$. Let $W_c$ be the concatenation of these walks of total length $t$.

We apply Lemma~\ref{lem:felipe} with $Z_i:=X_i^*(t')/\sigma _{t'}$ and $S_n:=\sum _{i=1}^nZ_i=X_c(nt')/\sigma _{t'}$. Note   that $\mathbb E [Z_i]=0$ and $\cov (Z_i)=\Id$. Moreover, by Lemma~\ref{lem:var} together with Assumptions~$\Hyptwo{t'}$ and $\Hyptwo{t-nt'}$ we have that $\mathbb E [\|Z_i\|^3]\le C \log^{24}(1/p)\ll r/\sigma _{t'}$. Thus, using that $\sigma _{t'} \ge c\sqrt{t'/p}\ge ct^{1/3}p^{-2/3}$ by Lemma~\ref{lem:var} we obtain using Lemma~\ref{lem:felipe} that
\[\mathbb P \big( \|X_c(nt')-z\| \le 3r \big)= \mathbb P \big( \|S_n-z/\sigma _{t'}\| \le 3r/\sigma _{t'} \big) \le C  (r/\sigma _{t'})^dn^{-d/2} \le Cr^d(t/p)^{-d/2}.\]
Thus, by Assumption~$\Hyptwo{t-nt'}$ we have
\[\mathbb P \big( \|X_c(t)-z\| \le 2r \big)\le \mathbb P \big( \|X_c(nt')-z\| \le 3r \big) +\mathbb P \big( \|X^*_{n+1}(t-nt')\| \le r \big) \le Cr^d(t/p)^{-d/2} .\]
This completes the proof of the corollary using Corollary~\ref{cor:conca} since the error from concatenation is $(n+1)p^{-1}\log ^7(1/p)$ which is asymptotically smaller than $r$.
\end{proof}

\appendix
\section{An elementary anticoncentration estimate}

In this section we give the proof of Lemma~\ref{lem:felipe}, which is based on Fourier analysis. We define the Fourier transform $\Ft{\nu}$ of a function $\nu$ by $\Ft{\nu}(\xi)=\int_{\mathbb R^d} e^{-2\pi i\xi \cdot x} d\nu(x)$. We start with the following elementary lemma.
\begin{lemma}\label{lem:felipe2}
Let $\nu$ be a probability measure on $\Real^d$. Then for all $r>0$, we have
\[
\sup_{z\in\Real^d} \Prob_\nu (Y \in B_{r^{-1}}(z)) \leq C r^{-d} \int_{B_r} |\Ft{\nu}(\xi)| \diff \xi.
\]
\end{lemma}

\begin{proof}
By translation invariance in form of $|\Ft{\nu(z+\cdot)}|=|\Ft{\nu}|$, it suffices to prove that
\[
\Prob_\nu (Y\in B_{r^{-1}}(0)) \leq C r^{-d} \int_{B_r} |\Ft{\nu}(\xi)| \diff \xi.
\]

Let $\rho$ be a positive function with Fourier transform compactly supported $\operatorname*{supp} \Ft{\rho} \subset B_1$, with lower bound $\rho(x)\geq c$ for $x\in B_1$, and normalized so that $\int \rho = 1$.

Let $\rho_r(x) = r^d \rho(rx)$ which satisfies $\int \rho_r = 1$, $\rho_r(x)\geq cr^d $ on $B_{r^{-1}}$,
and $\operatorname*{supp} \Ft{\rho} \subset B_r$.   Then by Plancherel's formula,
\[
\int \rho_r(x) \diff\nu(x) = \int \Ft{\rho_r}(\xi) \Ft{\nu}(\xi) \diff\xi
\leq \int_{B_r} |\Ft{\nu}(\xi)| \diff \xi .
\]
On the other hand,
\[
\int \rho_r(x)\diff \nu (x)\geq  \int_{B_{r^{-1}}} cr^d \diff \nu(x)= cr^d\cdot \Prob(Y\in B_{r^{-1}}),
\]
as needed.
\end{proof}
We conclude with the proof of Lemma~\ref{lem:felipe}.
Let $\mu$ be the probability measure of $Z$.  The fact that $\mu$ is a probability measure and the moment conditions on $Z$ imply the following:
\begin{equation*}
\Ft{\mu}(0) = 1, \quad
\partial_j \Ft{\mu}(0) = 0, \quad
\partial_{jk} \Ft{\mu}(0) = \delta_{jk}, \quad
\sup_{\xi} |\partial_{ijk} \Ft{\mu}(\xi)| \leq C_d M.
\end{equation*}

By the Taylor remainder theorem the above imply
\[
|\Ft{\mu}(\xi) - (1 - \frac12 |\xi|^2)| \leq C_d M |\xi|^3.
\]
In particular, for $|\xi| \leq (4C_d M)^{-1}$ we have
\[
|\Ft{\mu}(\xi)| \leq 1 - \frac12 |\xi|^2 + C_d M |\xi|^3
\leq 1 - \frac14 |\xi|^2 \leq \exp(-\frac14 |\xi|^2).
\]
The probability measure of $S_n$ is the $n$-fold convolution $\mu^{\ast n}$. It has Fourier transform
$\Ft{\mu}^n(\xi)$, which satisfies
\[
|\Ft{\mu}^n(\xi)|\leq \exp(-\frac{n}{4} |\xi|^2)
\]
on $|\xi|\leq (4C_d M)^{-1}$.   In particular, with $r= (4C_dM)^{-1}$ we have that
\[
\int_{B_r} |\Ft{\mu}(\xi)|^n \diff\xi \leq
\int_{\Real^d} \exp(-\frac{n}{4} |\xi|^2) \leq C n^{-d/2}.
\]
Now the result follows from the Lemma~\ref{lem:felipe2} applied to $\nu=\mu^{\ast n}$.

\newpage

\section{Summary of Notation}
Below we collect the definitions of various objects used throughout the proof for the convenience of the reader.

\begin{figure}[h!]
\centering
\renewcommand{\arraystretch}{1.5}
\begin{tabular}{l|c|r}
Notation & Meaning & Reference \\
\hline
$\tilde{W}$ & The driving process & Section~\ref{sec:driving} \\
$\tau_{\rm int}$ & The time of first interaction & Equation~\eqref{tau-int} \\
$\tau_{\rm clo}$ & First time walk forms a closed cycle & Equation~\eqref{tau-clo} \\
$W(t)$ & Regenerated mirror walk starting at $(0,e_1)$ & Section~\ref{sec:regenerated} \\
$\alpha(t)$ & Last regeneration time before $t$ & \\
$W^*(t)$ & Regenerated walk with random initial velocity & \\
$\mcal{M}$ & Set of mirror locations & Equation~\eqref{eq:M-def} \\
$t_*$ & Kinetic time unit $\lfloor p^{-1}\log^3(1/p)\rfloor $& Section~\ref{sec:prelims} \\[8pt]
$\tau_{\rm few}$ &
\renewcommand{\arraystretch}{0.9}
\begin{tabular}{c} Stopping time at which path continues \\ straight for a long time \end{tabular} &  \\[9pt]
$\tau_{\rm many}$ &
\renewcommand{\arraystretch}{0.9}
\begin{tabular}{c}
Stopping time at which many collisions
\\ are discovered in one short interval
\end{tabular} &  \\
$\Hypone{t}$,$\Hyptwo{t}$ & The inductive hypotheses & Section~\ref{sec:hypotheses} \\
$\mcal{H}(T)$ & $\Hypone{t}$ and $\Hyptwo{t}$ hold for all $0\leq t\leq T$
& Section~\ref{sec:hypotheses} \\
$\bar{T}$ & The long timescale $\lfloor e^{\log^2(1/p)}\rfloor$ & \\[9pt]
$\mcal{S}$ &
\renewcommand{\arraystretch}{0.9}
\begin{tabular}{c}
The set of times at which the walk may change direction \\
and head towards a previously discovered mirror
\end{tabular}
& Equation~\eqref{eq:S} \\[10pt]
$\tau_{\rm hit}$ &
\renewcommand{\arraystretch}{0.9}
\begin{tabular}{c}
The time after $t$ of the first interaction \\
of the walk with its past before $t$
\end{tabular} & Equation~\eqref{eq:def_tau_hit} \\[8pt]
$\mcal{R}$ & The set of relaxed times & Definition~\ref{eq:def_relax} \\
$\tau_{\rm hea}$ & The first time at which a heavy block forms
& Equation~\eqref{eq:tau-hea-def} \\[9pt]
$\tau_{\rm rel}$ &
\renewcommand{\arraystretch}{0.9}
\begin{tabular}{c}The first time there are not many \\ relaxed
times in a short interval\end{tabular} & Equation~\eqref{eq:tau-rel} \\[9pt]
$W_c=(X_c,V_c)$ & The concatenation of two independent
regenerated walks & Section~\ref{sec:conc}
\end{tabular}
\end{figure}

\section*{Acknowledgments}
The authors are grateful to Tom Spencer for suggesting that we team up on this project.
DE thanks Allan Sly for fruitful discussions at an early stage of this project.
AG acknowledges the hospitality of Stanford University as this work was coming to an end, and financial support from the European Research Council (ERC) under the European Union's Horizon 2020 research and innovation programme (Grant Agreement n$^\circ$864066)\footnote{Views and opinions expressed are however those of the authors only and do not necessarily reflect those of the European Union or the European Research Council Executive Agency. Neither the European Union nor the granting authority can be held responsible for them.}.
FH is supported by NSF grant DMS-2303094.

\bibliographystyle{alpha}

\bibliography{mirror}

\end{document}